\documentclass[a4paper,oneside,11pt]{article} 

\usepackage[utf8]{inputenc} 
\usepackage[T1]{fontenc} 
\usepackage{textcomp} 
\usepackage[english]{babel} 
\usepackage{indentfirst}
\usepackage[autolanguage]{numprint} 
\usepackage{hyperref} 
\hypersetup{colorlinks=true,linkcolor=blue,citecolor=red,urlcolor=blue}
\usepackage{graphicx} 
\usepackage{verbatim} 
\usepackage{makeidx} 
\usepackage[refpage]{nomencl} 
\usepackage{enumitem} 
\usepackage{tikz} 
\usetikzlibrary{matrix,arrows,patterns}
\usepackage{multicol}
\usepackage{fullpage}
\usepackage{comment}
\usepackage{stmaryrd}
\usepackage{mathrsfs}
\usepackage{ulem}

\newcommand\mscriptsize[1]{\mbox{\scriptsize\ensuremath{#1}}}

\usepackage{amsmath,amssymb,amsthm,amsfonts} 
\usepackage[all]{xy} 

\numberwithin{equation}{section}

\newcounter{Main}

\theoremstyle{plain} 

\swapnumbers 
\theoremstyle{definition} 
\newtheorem{Def}{Definition}[section] 
\newtheorem{Def,Thm}[Def]{Definition and theorem} 
\newtheorem{Def,Prop}[Def]{Proposition-definition} 

\theoremstyle{plain} 
\newtheorem{Proposition}[Def]{Proposition} 
\newtheorem{Lemma}[Def]{Lemma} 
\newtheorem{Theorem}[Def]{Theorem} 
\newtheorem{Corollary}[Def]{Corollary} 
\theoremstyle{remark} 
\newtheorem{Example}[Def]{Example} 
\newtheorem{Remark}[Def]{Remark} 

\usepackage{xcolor}
\usepackage[all]{xy}
\newcommand{\V}{\mathrm{u}_{2\mathtt{a}}}
\newcommand{\U}{\mathrm{u}_{\mathtt{a}}}

\newcommand{\X}{\mathscr{X}}
\newcommand{\Aa}{\mathscr{A}}

\newcommand{\Conv}{\mathop{\scalebox{1.5}{\raisebox{-0.2ex}{$\ast$}}}}

\newcommand\bmattrix[3]{\textnormal{\scriptsize$\left(\begin{array}{c}#1\\#2\\#3\end{array}\right)$\normalsize}}
\newcommand\sbmattrix[9]{\textnormal{\scriptsize$\left(\begin{array}{ccc}#1&#2&#3\\#4&#5&#6\\#7&#8&#9\end{array}\right)$\normalsize}}

\newcommand\ssbmatrix[4]{\textnormal{\scriptsize$\left(\begin{array}{cc}#1&#2\\#3&#4\end{array}\right)$\normalsize}}

\title{On the homology of special unitary groups \\ over polynomial rings} 
\author{Claudio Bravo\footnote{Instituto de Matem\'aticas, Universidad de Talca, Talca, Chile.
Email address: claudio.bravo@utalca.cl.}} 

\date{}

\begin{document} 

\maketitle

\begin{abstract}
In this work, we answer the homotopy invariance question for the ``smallest'' non-isotrivial group-scheme over $\mathbb{P}^1$, obtaining a result, which is not contained in previous works due to Knudson and Wendt.
More explicitly, let $\mathcal{G}=\mathrm{SU}_{3,\mathbb{P}^1}$ be the (non-isotrivial) non-split group-scheme over $\mathbb{P}^1$ defined from the standard (isotropic) hermitian form in three variables.
In this article, we prove that there exists a natural homomorphism $\mathrm{PGL}_2(F) \to \mathcal{G}(F[t])$ that induces isomorphisms $H_*(\mathrm{PGL}_2(F), \mathbb{Z}) \to H_*(\mathcal{G}(F[t]), \mathbb{Z})$.
Then we study the rational homology of $\mathcal{G}(F[t,t^{-1}])$, by previously describing suitable fundamental domains for certain arithmetic subgroups of $\mathcal{G}$.\\
\textbf{MSC codes:} primary 20G10, 20G30, 20E08; secondary 11E57, 20F65.\\
\textbf{Keywords:} Homology, special unitary groups, arithmetic subgroups and Bruhat-Tits trees.
\end{abstract}

\section{Introduction}\label{section introduction}

The fundamental theorem of algebraic K-theory states that for each regular ring $R$ there is a natural isomorphism between the $i$-th K-theory groups $K_i(R[t]) \cong K_i(R)$, for all $i\geq 0$ (cf. \cite[Th. 8, \S 6, Ch. 8]{Ktheory}).
In a more general language, a presheaf $\mathcal{F}$ on the category of schemes over a base $\mathcal{S}$ is called homotopy invariant whenever $\mathcal{F}(X \times_{\mathcal{S}} \mathbb{A}^1_\mathcal{S})\cong \mathcal{F}(X)$, for any scheme $X$ over $\mathcal{S}$.
These homotopy invariant presheaves have been studied and applied to several domains that go beyond K-theory (cf. \cite{SuslinKtheory1,SuslinVoevodsky,Voevodsky}).

Since the first K-theory group of a ring $R$ is isomorphic to $\mathrm{GL}(R)^{\mathrm{ab}}=H_1\big( \mathrm{GL}(R), \mathbb{Z} \big)$, where $\mathrm{GL}(R)$ is the limit direct $\varinjlim \mathrm{GL}_n(R)$ given by the inclusions
$\mathrm{GL}_n(R) \hookrightarrow \mathrm{GL}_{n+1}(R)$, $A \mapsto \ssbmatrix{A}{0}{0}{1},$ the fundamental theorem of K-theory implies that
$H_1 \big( \mathrm{GL}(R), \mathbb{Z} \big) \cong H_1 \big( \mathrm{GL}(R[t]), \mathbb{Z} \big).$
The existence of homotopy invariance results in K-theory motivated the research on unstable versions for the homology of algebraic groups.
In the latter, the homotopy invariance question asks under which conditions, on a linear algebraic group $\mathcal{G}$ and a ring $R$, the group homomorphism $\mathcal{G}(R) \to \mathcal{G}(R[t])$ induces isomorphisms:
$$H_*\big(\mathcal{G}(R),M\big) \xrightarrow{\cong} H_*\big(\mathcal{G}(R[t]),M\big),$$
in the homology of the involved groups.

In the current literature, there exist some interesting results on this problem.
On the one hand, assuming that $F$ is a finite field of characteristic $p>0$ and that $\mathcal{G}$ is a semisimple simply connected split $F$-group, Soul\'e gives in \cite{Soulé} a positive answer to the question of homotopy invariance when $M$ is a field with $\mathrm{char}(M) \neq 0$ and $(\mathrm{char}(M),p)=1$.
In order to prove this result, Soul\'e studies the action $\mathcal{G}(F[t])$ on the Bruhat–Tits building $X\big(\mathcal{G}, F(\!(t^{-1})\!)\big)$ by determining a fundamental domain of this action.
On the other hand, assuming that $F$ is a field of characteristic $0$, Knudson in \cite{Knudson1} obtains various interesting results on the integral homology of $\mathrm{SL}_2$.
The author, for instance, proves the following result:

\begin{Theorem}\cite[Th. 1.3]{Knudson1}\label{teo knudson hom equiv}
Let $F$ be a field with $\mathrm{char}(F)=0$. The canonical inclusion $\mathrm{SL}_2(F) \hookrightarrow \mathrm{SL}_2(F[t])$ induces isomorphisms $H_*\big(\mathrm{SL}_2(F),\mathbb{Z}\big) \xrightarrow{\cong} H_*\big(\mathrm{SL}_2(F[t]),\mathbb{Z}\big)$.
\end{Theorem}

The proof of Th. \ref{teo knudson hom equiv} is essentially written in the next $2$ steps.
First, Knudson describes the (homological) Mayer-Vietoris sequence defined by Nagao's decomposition $\mathrm{SL}_2(F[t]) \cong \mathrm{SL}_2(F) \ast_{B(F)} B(F[t])$, where $B$ is the subgroup of upper triangular matrices in $\mathrm{SL}_2$.
Then he reduces the homology groups $H_*\big(B(F), \mathbb{Z}\big)$ and $H_*\big(B(F[t]), \mathbb{Z}\big)$ to $H_*(F^{*}, \mathbb{Z})$
by using the Hochschild-Serre spectral sequences associated to the decomposition $B(R') \cong F^* \ltimes R'$, for $R'=F[t]$ and $R'=F$.

Since the algebraic K-theory satisfies $K_i(R[t,t^{-1}]) \cong K_i(R) \oplus K_{i-1}(R)$, for all $i>0$, it is also interesting to investigate unstable versions for the homology of algebraic group.
Going to this question, Knudson in \cite{Knudson1} describes the integral homology of $\mathrm{SL}_2(F[t,t^{-1}])$ via the following result:

\begin{Theorem}\cite[\S 5]{Knudson1}\label{teo knudson hom equiv for SL2(F[t,t-1])}
If $F$ is a field with $\mathrm{char}(F)=0$, then there is an exact sequence:
$$\cdots \to H_n\big(F^*, \mathbb{Z} \big) \to H_n\big(\mathrm{SL}_2(F), \mathbb{Z} \big) \oplus H_n\big(\mathrm{SL}_2(F),\mathbb{Z} \big) \to H_n\big(\mathrm{SL}_2(F[t,t^{-1}]), \mathbb{Z}\big) \to \cdots $$
\end{Theorem}

In order to prove Th. \ref{teo knudson hom equiv for SL2(F[t,t-1])}, Knudson describes the action of $\tilde{\Gamma}:=\mathrm{SL}_2(F[t,t^{-1}])$ on the product of the Bruhat-Tits trees of $\mathrm{SL}_2\big(F((t^{\epsilon}))\big)$, for $\epsilon\in \lbrace 1,-1\rbrace$.
In this way, he obtains the amalgamated product $\tilde{\Gamma} \cong \Gamma \ast_{\Gamma_0} \Gamma$, where $\Gamma=\mathrm{SL}_2(F[t])$ and $\Gamma_0$ is the Hecke congruence subgroup
$$\Gamma_0=\Gamma_0(t)=\left\lbrace \ssbmatrix{a}{b}{tc}{d} \big\vert a,b,c,d \in F[t], \, ad-tcb=1 \right\rbrace.$$
Then, Knudson proves that $H_*(\Gamma, \mathbb{Z}) \cong H_*(F^*,\mathbb{Z})$.
Thus, Th. \ref{teo knudson hom equiv for SL2(F[t,t-1])} essentially follows from the Mayer-Vietoris sequence defined by $\tilde{\Gamma} \cong \Gamma \ast_{\Gamma_0} \Gamma$.
Then, applying Th. \ref{teo knudson hom equiv for SL2(F[t,t-1])} and \cite[Th. 2.1]{BorelYangContCoh}, Knudson proves in \cite[Th. 5.1]{Knudson1} that
if $F$ is a number field with $r$ (resp. $s$) real embeddings (resp. conjugate pairs of complex embeddings), then there are (natural) isomorphisms $H_p\big(\mathrm{SL}_2(F[t,t^{-1}]),\mathbb{Q}\big) \xrightarrow{\cong} H_{p-1}(F^{*},\mathbb{Q})$, for all $p\geq 2r+3s+2$.

Going to the case of higher rank algebraic $F$-groups, Knudson in \cite{Knudson2} extends the Soul\'e’s approach and deduces various homotopy invariance results for $\mathrm{SL}_n$ over arbitrary infinite fields $F$ and $M=\mathbb{Z}$.
These results generalize both Th. \ref{teo knudson hom equiv} and Th. \ref{teo knudson hom equiv for SL2(F[t,t-1])}.

Still in the context of algebraic $F$-groups, 
one of the strongest existing results on the homotopy invariance question was proved by 
Wendt in \cite{Wendt}. It states the following:

\begin{Theorem}\cite[Theo. 1.1]{Wendt}\label{teo wendt isotrivial}
Let $F$ be an infinite field and let $\mathcal{G}$ be a connected reductive smooth linear algebraic $F$-group.
The canonical inclusion $F \hookrightarrow F[t]$ induces isomorphisms $H_*\big(\mathcal{G}(F),\mathbb{Z}\big) \xrightarrow{\cong} H_*\big(\mathcal{G}(F[t]),\mathbb{Z}\big)$ if the order of the fundamental group of $\mathcal{G}$ is invertible in $F$.
\end{Theorem}


Some positive answers for the invariance homotopy question on others arithmetic groups, different to $\mathcal{G}(F[t])$, are given in works as either \cite[Ch. 4]{KnudsonBook} or \cite{Hutchinson, Knusdonelementarygroups, Knudsonfinitefields}.
Going to the case of arbitrary regular rings, it follows from \cite{KrMc} that homotopy invariance does not work for $H_1$ over $\mathcal{G}(R[t])$, when $\mathrm{rk}(\mathcal{G})=1$ and $R$ is an integral domains which is not field.
In the same context, when $\mathrm{rk}(\mathcal{G})=2$, homotopy invariance fails for $H_2$ as discussed in \cite{Wendt2}.
It therefore seems that one cannot hope for an extension of the preceding results for arbitrary regular rings or even polynomial rings in more than $1$ variable.

Note that almost all previous results are specific to algebraic $F$-groups, also called isotrivial groups.
Thus, it is natural to seek for extensions to group schemes defined over projective algebraic $F$-curves, but not over $F$.
Since each split semisimple group has a $\mathbb{Z}$-model (called its Chevalley model), in order to study the homotopy invariance question in this new context, we have to go to the case of non-isotrivial quasi-split groups.
Indeed, this paper is devoted to present a first advance in this direction for the ``smallest'' non-isotrivial quasi-split group.
More specifically, in this article we extend Th. \ref{teo knudson hom equiv} and \ref{teo knudson hom equiv for SL2(F[t,t-1])} to a certain special unitary group $\mathrm{SU}_{3,\mathbb{P}^1_F}$ defined over $\mathbb{P}^1_F$, which does not have an $F$-structure (See \S \ref{section main results}).

Theorems \ref{teo knudson hom equiv} and \ref{teo wendt isotrivial} admit the following interpretation in terms of group actions.
Indeed, in these hypotheses, $\mathcal{G}(F[t])$ acts on the Bruhat-Tits building $X=X\big(\mathcal{G}, F(\!(t^{-1})\!)\big)$ with a sector chamber $Q_0$ (see \cite[\S 1.4]{AbramenkoBrown} for its definition) as a fundamental domain according to \cite[Th. 1]{Soulé} and \cite[Th. 2.1]{Margaux}.
Moreover the $\mathcal{G}(F[t])$-stabilizer of the tip $v_0$ of $Q_0$ is $\mathcal{G}(F)$.
Thus, Thms \ref{teo knudson hom equiv} and \ref{teo wendt isotrivial} can be paraphrased by saying that the $\mathbb{Z}$-homology of $\mathcal{G}(F[t])$ reduces to the homology of $\mathrm{Stab}_{\mathcal{G}(F[t])}(v_0)$.
In Th. \ref{main teo 1} we prove that this interpretation holds for the non-isotrivial group scheme $\mathrm{SU}_{3,\mathbb{P}^1_F}$.
Moreover, we conjecture that, for each semi-simple simply connected non-isotrivial quasi-split groups schemes $\mathcal{G}_{\mathbb{P}^1_F}$ defined over $\mathbb{P}^1_F$, the group $\Gamma=\mathcal{G}_{\mathbb{P}^1_F}(F[t])$ acts on $X$ with a sector chamber $Q_0$ as a fundamental domain, and that its integral homology reduces to the $\mathbb{Z}$-homology of $\mathrm{Stab}_{\Gamma}(v_0)$, where $v_0$ is the tip of $Q_0$ as above.
Since $\mathrm{SU}_3$ and $\mathrm{SL}_2$ encode the behavior of root subgroups of quasi-split reductive groups, we hope that this work, combined with the approach given in \cite{Wendt}, inspires a future study in this direction.



\section{Main results}\label{section main results}

In order to introduce our main results, we consider the $\mathbb{P}^1_{F}$-group scheme $\mathcal{G}$ defined as follows.
Assume that $\mathrm{char}(F)\neq 2$ and set the $2:1$ (ramified) cover $\psi: \mathcal{C}=\mathbb{P}^1_F \to \mathcal{D}=\mathbb{P}^1_F$ given by $z\mapsto z^2$.
This cover corresponds to the quadratic extension field $L=F(\sqrt{t})$ over $K=F(t)$.
Let $R$ be a subring of $K$ such that $\mathrm{Quot}(R)=K$, and let $S\subset L$ be its integral closure in $L$.
At any affine subset $\mathrm{Spec}(R)\subset \mathcal{C}$, we denote by $\mathcal{G}_R$ the special unitary group-scheme defined from the $R$-hermitian form:
\begin{equation}\label{eq h_R}
h_R:S^3\to R, \quad h_R(x,y,z):=x\bar{z}+ y \bar{y} + z \bar{x}.   
\end{equation}
Since we can cover $\mathcal{C}$ with affine subsets $\mathrm{Spec}(R_i)$ with affine intersection, the groups $\mathcal{G}_{R_i}$ can be glue in order to define the $\mathcal{C}$-group scheme $\mathcal{G}=\mathrm{SU}_{3,\mathcal{C}}$.
The generic fiber $\mathcal{G}_K$ of $\mathcal{G}$ is a quasi-split semisimple simply connected $K$-group of (split) rank $1$. 
As we show in \S \ref{subsection radical datum}, this group splits over $L$.
In particular, since $L$ does not have the form $F'(t)$, for some finite extension $F'/F$, the group scheme $\mathcal{G}$ is non-isotrivial, i.e., $\mathcal{G}$ does not have an $F$-model.
Moreover, at any closed point $P$ of $\mathcal{C}$ that fails to decompose at $\mathcal{D}$, the $K_P$-group $\mathcal{G}_{K_P}=\mathcal{G}_K \otimes_{K} K_P$ is quasi-split and it splits at the quadratic extension $L_P=L \otimes_K K_P$.
If $P$ decomposes at $\mathcal{D}$, then $\mathcal{G}_{K_P}=\mathrm{SL}_{3,K_P}$.

Let $\Gamma:=\mathrm{SU}_3(F[t])$ be the group of $R=F[t]$ points of $\mathcal{G}$.
This group can be represented as the group of matrices in $\mathrm{SL}_3(F[\sqrt{t}])$ preserving the form $h_{R}$.
Our first main result is the following, which extends the results of Knudson and Went, in Ths. \ref{teo knudson hom equiv} and \ref{teo wendt isotrivial}, to the context of the group scheme $\mathcal{G}=\mathrm{SU}_{3,\mathcal{C}}$.

\begin{Theorem}\label{main teo 1}
Let $F$ be a field with $\mathrm{char}(F)=0$.
There exists an injective (and natural) homomorphism $\iota: \mathrm{PGL}_2(F) \hookrightarrow \Gamma=\mathrm{SU}_3(F[t])$, which induces isomorphisms:
$$\iota_*: H_* \big( \mathrm{PGL}_2(F), \mathbb{Z} \big) \xrightarrow{\cong}  H_* \big( \mathrm{SU}_3(F[t]), \mathbb{Z} \big).$$
\end{Theorem}

Note that Th. \ref{main teo 1} relates the homology groups two different algebraic groups, namely $\mathrm{SU}_3$ and $\mathrm{PGL}_2$.
This result is proved in \S \ref{section homolofy of SU(F[t])}.
It essentially follows in two steps as the analogous result of Knudson in Th. \ref{teo knudson hom equiv}.
Indeed, we fist focus on the Mayer-Vietoris sequence defined by the amalgamated product:
\begin{equation}\label{eq nagao dec for SU3}
\Gamma\cong \mathrm{PGL}_2(F) \ast_{B_0} B, 
\end{equation}
where $B$ is the group of upper triangular matrices in $\Gamma$ and $B_0= \iota\big(\mathrm{PGL}_2(F)\big) \cap B$.
See Lemma \ref{lemma SO3 is PSL2} and \cite[Theo. 2.4]{ArenasBravoLoiselLucchini} for more details.
Then, we reduce the homology groups $H_*\big(B, \mathbb{Z}\big)$ and $H_*\big(B_0, \mathbb{Z}\big)$ to $H_*(F^{*}, \mathbb{Z})$.
Since the unipotent radical of $B$ is, in the language of homological algebra, a non-split extension of $F[\sqrt{t}]$ by itself, the latter reduction is more involved than its analog for the special linear groups described in \cite{Knudson1} \& \cite{Knudson2} or even its analog in the isotrivial case given in \cite{Wendt}.

Now, let us denote by $\tilde{\Gamma}=\mathrm{SU}_3\big(F[t,1/t]\big)$ the group of $R'=F[t,1/t]$ points of $\mathcal{G}$.
This group can be represented as the subgroup of $\mathrm{SL}_3\big(F[\sqrt{t}, 1/\sqrt{t}]\big)$ of matrices preserving $h_{R'}$.
By describing the action of $\tilde{\Gamma}$ on the Bruhat-Tits tree $X\big(\mathrm{SU}_3, F(\!(t)\!) \big)$, we prove in \S \ref{section amalgam} that $\tilde{\Gamma}$ decomposes as the amalgamated product $\Gamma \ast_{\Gamma_0} \hat{\Gamma}$, where
\begin{equation*}
\hat{\Gamma} := \mscriptsize{ \begin{pmatrix} F[\sqrt{t}] & F[\sqrt{t}] & (1/\sqrt{t})F[\sqrt{t}] \\ \sqrt{t}F[\sqrt{t}] & F[\sqrt{t}] & F[\sqrt{t}] \\ \sqrt{t}F[\sqrt{t}] & \sqrt{t}F[\sqrt{t}] & F[\sqrt{t}] \end{pmatrix}}\cap \tilde{\Gamma},
\end{equation*}
and $\Gamma_0=\Gamma_0(t)$ is the Hecke congruence subgroup $\Gamma \cap \hat{\Gamma}$.

In order to describe the integral homology of $\tilde{\Gamma}$, we focus on the description of $H_*(\hat{\Gamma},\mathbb{Z})$ and $H_*(\Gamma_0,\mathbb{Z})$. 
To do so, in \S \ref{section amalgam} and \S \ref{section homology of Hecke},
we describe a fundamental domain for the action of each group, $\hat{\Gamma}$ or $\Gamma_0$, on the Bruhat-Tits tree $X\big(\mathrm{SU}_3, F(\!(t^{-1})\!) \big)$.
Then, by using Bass-Serre theory, we prove in Th. \ref{teo dec of H1 as an amm product} that $\hat{\Gamma}$ is isomorphic to the free product of $\mathrm{SL}_2(F)$ with the group of upper triangular matrices $\hat{B}$ of $\hat{\Gamma}$, amalgamated along a maximal common subgroup $\hat{B}_0$.
This allows us to describe the integral homology of $\hat{\Gamma}$, as next result shows:

\begin{Theorem}\label{main teo hom h1}
Let $F$ be a field with $\mathrm{char}(F)=0$. There exists an injective (and natural) homomorphism $\mathrm{SL}_2(F) \hookrightarrow \hat{\Gamma}$ inducing isomorphisms $H_*\big(\mathrm{SL}_2(F), \mathbb{Z}\big) \xrightarrow{\cong} H_*\big(\hat{\Gamma}, \mathbb{Z}\big)$.
\end{Theorem}

In Th. \ref{teo amal dec for hecke}, we show that $\Gamma_0$ is isomorphic to two copies of $B$ amalgamated by $F^{*}$ according to the diagonal injection $F^{*} \to B$, by previously determining a fundamental domain for the action of $\Gamma_0$ on its associated tree (see Cor. \ref{coro fund reg for H01}).
This allows us to prove the following result.

\begin{Theorem}\label{main teo hom of h01}
Let $F$ be a field with $\mathrm{char}(F)=0$. The diagonal homomorphism $F^* \hookrightarrow \Gamma_0$ induces isomorphisms $H_*\big(F^{*},\mathbb{Z}\big) \xrightarrow{\cong} H_*\big(\Gamma_0,\mathbb{Z}\big)$. 
\end{Theorem}

Analogous results for some other relevant congruence subgroups of $\Gamma$ are described in \S \ref{section homology of Hecke}.
Since $\mathcal{G}_K=\mathrm{SU}_{3}(h_K)$ is simply connected, we conjecture that Ths. \ref{main teo 1}, \ref{main teo hom h1} \& \ref{main teo hom of h01} extend to the context where $F$ is an arbitrary infinite field with $\mathrm{char}(F)\neq 2$.
To do so, it could be interesting to use the approach of \cite{Wendt}, by previously studying the analog of Prop. \ref{prop inv hom for Bst} for arbitrary infinite field (for instance, by extending the method outlined in \cite[Ch. 2, \S 2.2]{KnudsonBook}).

The decomposition of $\tilde{\Gamma}$ as $\Gamma \ast_{\Gamma_0} \hat{\Gamma}$ yields a Mayer-Vietoris sequence on homology. Applying Th. \ref{main teo 1}, \ref{main teo hom h1} \& \ref{main teo hom of h01} to this Mayer-Vietoris sequence, we prove in \S \ref{subsection hom of hecke}, the next result on the rational homology of $\tilde{\Gamma}$, extending Th. \ref{teo knudson hom equiv for SL2(F[t,t-1])} to our context.

\begin{Theorem}\label{main teo 4}
For each field $F$ with $\mathrm{char}(F)=0$, there is an exact sequence of the form:
$$
\cdots \to H_n\big(F^*, \mathbb{Z} \big) \to H_n\big(\mathrm{PGL}_2(F), \mathbb{Z}\big) \oplus H_n\big(\mathrm{SL}_2(F),\mathbb{Z}\big) \to H_n\big(\mathrm{SU}_3(F[t,t^{-1}]), \mathbb{Z}\big) \to \cdots $$
\end{Theorem}

By using Th. \ref{main teo 4}, we prove in Cor. \ref{coro ab of tilde gamma} that $H_1 (\tilde{\Gamma}, \mathbb{Z} )=\lbrace 0 \rbrace$, i.e., that the abelianization of $\tilde{\Gamma}=\mathrm{SU}_3 \left(F\left[t, t^{-1}\right] \right)$ is trivial.
This extends a previous result due to Cohn in \cite{Cohn}.
By also applying Th. \ref{main teo 4}, we describe in \S \ref{subsection hom of hecke} the rational homology of $\mathrm{SU}_3(F[t,t^{-1}])$ over certian fields $F$.
This includes an extension of \cite[Th. 5.1]{Knudson1} to our context.
In fact, in Cor. \ref{coro hom SU3(F[t,t-1])}, we prove that if $F$ is a number field with $r$ (resp. $s$) real (resp. conjugate pairs of complex) embeddings of $F$, then the connecting map:
$$ \partial: H_n\big(\tilde{\Gamma}, \mathbb{Q}\big)= H_n \Big(\mathrm{SU}_3 \left(F\left[t, t^{-1}\right] \right), \mathbb{Q} \Big) \to H_{n-1}\big(F^{*}, \mathbb{Q} \big), $$
is injective for $n \geq 2r+3s+1$ and bijective for $n \geq 2r+3s+2$.

\section{Conventions and preliminaries}\label{section conv and prel}

In this section, we recall some general well-known facts on the structure of the algebraic group $\mathrm{SU}_3$, as well as its associated Bruhat-Tits tree.
In particular, we do not assume that $\mathrm{char}(F)=0$.

\subsection{The radical datum of \texorpdfstring{$\mathrm{SU}_3$}{SU3}}\label{subsection radical datum}

As in \S \ref{section introduction}, assume just that $\mathrm{char}(F)\neq 2$ and set the $2:1$ (ramified) cover $\psi: \mathcal{C}=\mathbb{P}^1_F \to \mathcal{D}=\mathbb{P}^1_F$ given by $z\mapsto z^2$.
Let $K=F(t)$ and $L=F(\sqrt{t})$ be the function fields of the curves $\mathcal{C}$ and $\mathcal{D}$ as defined above.
The quadratic extension $L/K$ is Galois.
In the sequel, we denote by $\overline{x}$ the image of $x\in L$ via the non-trivial element in $\mathrm{Gal}(L/K)$.

Let $\mathcal{G}_K=\mathrm{SU}(h_K)$ be the special unitary $K$-group defined by the hermitian form $h_K: L^3 \to K$, $h_K(x,y,z)=x\bar{z}+y\bar{y}+z\bar{x}$. This group can be represented as the subgroup of the Weil restriction $\mathrm{R}_{L/K}(\mathrm{SL}_{3,L})$ consisting in the elements preserving $h_K$.
Following \cite[4.1]{BT2} and \cite[§4, Case 2, p. 43--50]{Landvogt}, in this section, we recall some basic concepts on the radical datum of $\mathcal{G}_K$.
In order to do this, we write:
$$ \mathrm{diag}(x,y,z)= \sbmattrix{x}{0}{0}{0}{y}{0}{0}{0}{z}.$$
A maximal $K$-split torus $\mathcal{S}_K$ of $\mathcal{G}_K$ consists in the subgroup of diagonal matrices of the form $\mathrm{diag}(\lambda,1,\lambda^{-1})$, where $\lambda \in \mathbb{G}_{m,K}$.
The centralizer $\mathcal{T}_K$ of $\mathcal{S}_K$ in $\mathcal{G}_K$ is then a $K$-maximal torus of $\mathcal{G}_K$.
This group can be parameterized by $\mathrm{R}_{L/K}(\mathbb{G}_{m,L})$ via $\lambda \mapsto \mathrm{diag}\left( \lambda, \bar{\lambda}\lambda^{-1}, \bar{\lambda}^{-1} \right) \in \mathcal{T}_K$.
In particular, the group of $K$-point $\mathcal{T}_K(K)$ of $\mathcal{T}_K$ admits the following parametrization:
$$\Tilde{a}: L^{*} \to \mathcal{T}_K(K), \quad \lambda \mapsto \mathrm{diag}\left(\lambda, \bar{\lambda}\lambda^{-1}, \bar{\lambda}^{-1} \right).$$
Note that $\mathcal{T}_K$ splits over $L$, but no over $K$.
More explicitly, note that $\mathcal{T}_L:=\mathcal{T}_K \otimes_{K} L \cong \mathbb{G}_{m,L}^2$, however $\mathcal{T}_K \not\cong \mathbb{G}_{m,K}^2$. 
A basis of characters of the $L$-split torus $\mathcal{T}_L$ consists in the characters $\lbrace \alpha, \overline{\alpha} \rbrace$ defined by $\alpha(\mathrm{diag}(x,y,z))=yz^{-1}$ and $\overline{\alpha}(\mathrm{diag}(x,y,z))=xy^{-1}$.
These characters are related via the action of $\mathrm{Gal}(L/K)$.
We denote by $\texttt{a}$ (resp. $2\texttt{a}$) the restriction of $\alpha$ (resp. $\alpha+\overline{\alpha}$) to $\mathcal{S}_K$.
The root system of $\mathcal{G}_K$ is $\Phi=\lbrace \pm \texttt{a}, \pm 2\texttt{a} \rbrace$.
Here, the root $\texttt{a}$ generates the $\mathbb{Z}$-module of characters $X^*(\mathcal{S}_K)$, while $(2\texttt{a})^{\vee}$ generated the $\mathbb{Z}$-module of cocharacter $X_*(\mathcal{S}_K)$.

The Weyl group $W=\mathcal{N}_{\mathcal{G}_K}(\mathcal{S}_K)/\mathcal{Z}_{\mathcal{G}_K}(\mathcal{S}_K)$ of $\mathcal{G}_K$ has order $2$. 
This group acts on $\Phi$ by exchanging $\texttt{a}$ with $-\texttt{a}$ (resp. $2\texttt{a}$ with $-2\texttt{a}$) via its non-trivial element $w\in W$.
Moreover, a lift $\mathrm{s}\in \mathcal{G}_K$ of $w\in W$ is given by the matrix:
\begin{equation}\label{eq s}
\mathrm{s}:=\sbmattrix{0}{0}{-1}{0}{-1}{0}{-1}{0}{0}.
\end{equation}

A system of positive root $\Phi^{+} \subset \Phi$ consists in the set $\lbrace \texttt{a}, 2\texttt{a} \rbrace$.
This election fixes a $K$-Borel subgroup of $\mathcal{G}_K$.
Explicitly, $\mathcal{B}_K$ is the group of upper triangular matrices in $\mathcal{G}_K$.
The root subgroups $\mathcal{U}_{\texttt{a}}=\mathcal{U}_{\texttt{a},K}$ and $\mathcal{U}_{2\texttt{a}}=\mathcal{U}_{2\texttt{a},K}$ of $\mathcal{B}_K$ defined by $\texttt{a}$ and $2\texttt{a}$ are parameterized by:
\begin{equation}\label{eq u_a}
\begin{array}{cccc}
    \U :& \mathcal{H}(L,K) &\to  & \mathcal{U}_{\texttt{a}}\\
     & (u,v) &\mapsto & \mscriptsize{ \begin{pmatrix} 1 & -\bar{u} & v \\ 0 & 1 & u \\ 0 & 0 & 1\end{pmatrix}} 
     \end{array}
    \, \, \text{,  and } \, \, 
    \begin{array}{cccc}
    \V :& \mathcal{H}(L,K)^0 &\to  & \mathcal{U}_{2\texttt{a}}\\
     & v &\mapsto & \mscriptsize{ \begin{pmatrix} 1 & 0 & v \\ 0 & 1 & 0 \\ 0 & 0 & 1\end{pmatrix}} 
\end{array},
\end{equation}
where $\mathcal{H}(L,K)$ and $\mathcal{H}(L,K)^0$ are the following $K$-varieties defined from the norm and the trace $\mathrm{N}=\mathrm{N}_{L/K}$ and $\mathrm{Tr}=\mathrm{Tr}_{L/K}$ of $L/K$:
\begin{align*}
\mathcal{H}(L,K) &:= \left\{ (u,v) \in \mathrm{R}_{L/K}(\mathbb{G}_{a,L})^2 \big\vert \, \mathrm{N}(u) + \mathrm{Tr}(v) = 0\right\},\\
\mathcal{H}(L,K)^0 &:= \left\{ v \in \mathrm{R}_{L/K}(\mathbb{G}_{a,L}) \big\vert \, \mathrm{Tr}(v) = 0\right\}.
\end{align*}
In the sequel, we denote by $H(L,K)$ (resp. $H(L,K)^0$) the set of $K$-point of $\mathcal{H}(L,K)$ (resp. $\mathcal{H}(L,K)^0$).
By abuse of notation, we write $\U$ (resp. $\V$) for the induced parametrization $\U: H(L,K) \to \mathcal{U}_{\texttt{a}}(K)$ (resp. $\V: H(L,K)^0 \to \mathcal{U}_{2\texttt{a}}(K)$.
Thus, we can write:
$$\mathcal{B}_K(K)= \left\{ \U(x,y)\tilde{a}(\lambda)= \sbmattrix{\lambda}{-\bar{\lambda}\lambda^{-1}\bar{x}}{\lambda^{-1}y}{0}{\bar{\lambda}\lambda^{-1}}{\lambda^{-1}x}{0}{0}{\bar{\lambda}^{-1}} \Big| \lambda\in L^{*}, \, (x,y) \in H(L,K) \right\}.$$

The unipotent subgroups of $\mathcal{G}_K$ corresponding to the negative root $-\texttt{a}$ (resp. $-2\texttt{a}$) is parameterized by $\mathcal{H}(L,K)$ (resp. $\mathcal{H}(L,K)^0$) via the homomorhism $\mathrm{u}_{-\texttt{a}}:=\mathrm{s} \cdot \U \cdot \mathrm{s}$ (resp. $\mathrm{u}_{-2\texttt{a}}:=\mathrm{s} \cdot \V \cdot \mathrm{s}$).
In the sequel, we also we write $\mathrm{u}_{-\texttt{a}}$ (resp. $\mathrm{u}_{-2\texttt{a}}$) for the induced parametrization $H(L,K) \to \mathcal{U}_{-\texttt{a}}(K)$ (resp. $H(L,K)^0 \to \mathcal{U}_{-2\texttt{a}}(K)$) at the level of $K$-points.

\subsection{The Bruhat-Tits tree of \texorpdfstring{$\mathrm{SU}_3$}{SU3}}\label{subsection BT}

The main goal of this section is to write a brief introduction to the Bruhat-Tits building $\X_P$ defined by $\mathcal{G}_{K_P}=\mathcal{G}_K \otimes_K K_P$, where $P$ is a closed point of $\mathbb{P}^1_F$, which is non-split (inert) at $L$.
Indeed, recall that each closed point $P \in \mathbb{P}^1_F$ gives rise to a valuation $\nu_P$ on $K$. 
Its completion $K_P$ is endowed with a valuation map, which, by abuse of notation, we also denote by $\nu_P$.
Since $P$ is assumed non-split over $L$, the field $L_P=L\otimes_K K_P$ is a quadratic extension of $K_P$, and, in particular, it is a discrete valued field.
By abuse of notation, we denote by $\nu_P$ the valuation on $L_P$ extending the valuation on $K_P$.
In the sequel, we write $\mathrm{Gal}(L_P/K_P)=\lbrace \mathrm{id}, \overline{(\cdot)} \rbrace$.

The (standard) apartment $\Aa_{P}$ of $\X_{P}$ defined by the maximal $K_P$-torus $\mathcal{S}_{K_P}:=\mathcal{S}_K \otimes_K K_P$ is, by definition:
$$\Aa_{P}=X_*(\mathcal{S}_{K_P}) \otimes_{\mathbb{Z}} \mathbb{R},$$
where $X_*(\mathcal{S}_{K_P})$ is the module of cocharacters of $\mathcal{S}_{K_P}$, as in \S \ref{subsection radical datum}.
The vertex set $\mathrm{V}(\Aa_P)$ of $\Aa_P$ is the set of elements $x=r \texttt{a}^{\vee}$ such that $\texttt{a}(x)=2r$ belongs to $\Gamma_a:= \frac{1}{4} \mathbb{Z}$.
The vertex set $\mathrm{V}(\Aa_P)$ induces a tessellation on the apartment.
An edge in $\Aa_P$ is an open segment in $\Aa_P$ joining two consecutive vertices.
We denote by $\mathrm{E}(\Aa_P)$ the edge set of $\Aa_P$.
See~\cite[4.2.21(4) and 4.2.22]{BT2} for more details. 
The group $N$ of $K_P$-points of $\mathcal{N}_{\mathcal{G}_{K_P}}(\mathcal{S}_{K_P})$ decomposes as $T \cup \mathrm{s}\cdot T$, where $T$ is the group of $K_P$-points of $\mathcal{T}_{K_P}=\mathcal{T}_K \otimes_K K_P$.
Since the group $T$ is parametrized by $L_P^*$ via $\lambda\mapsto \mathrm{diag}\left(\lambda,\bar{\lambda}\lambda^{-1}, \bar{\lambda}^{-1}\right)$, the group $N$ acts on $\Aa_P$ via:
$$ \tilde{a}(\lambda) \cdot x = x- \frac{1}{2}\nu_P(\lambda)\texttt{a}^{\vee}, \text{ and } \mathrm{s} \cdot x=-x, \quad  \forall x \in\Aa_P, \, \,  \, \forall \lambda \in L_P^{*}.$$
See~\cite[6.2.10]{BT1} and~\cite[4.2.7]{BT2} for more details. We write $H(L_P,K_P)$ (resp. $H^0(L_P,K_P)$) in order to denote the set of $K_P$-points of $\mathcal{H}(L,K) \otimes_K K_P$ (resp. $\mathcal{H}^{0}(L,K) \otimes_K K_P$).
By abuse of notation, we respectively denote by $\mathrm{u}_{\texttt{a}}$ and $\mathrm{u}_{2\texttt{a}}$ the maps from $H(L_P,K_P)$ and $H^0(L_P,K_P)$ to the root subgroups $\mathcal{U}_{\texttt{a}}(K_{P})$ and $\mathcal{U}_{2\texttt{a}}(K_{P})$.
For each $b\in \lbrace \texttt{a},-\texttt{a} \rbrace$, we set:
$$\mathcal{U}_{b,x}(K_P):=\big\lbrace \mathrm{u}_{b}(x,y) : (x,y) \in H(L_P,K_P), \,\, \nu_P(y) \geq - 2b(x) \big\rbrace.$$
Let us write $\mathcal{U}_{x}(K_P)= \big\langle \mathcal{U}_{\texttt{a}, x}(K_P), \,\mathcal{U}_{-\texttt{a}, x}(K_P) \big\rangle$, and
define $\backsim$ as the equivalence relation on $\mathcal{G}_{K_P}(K_P) \times \Aa_P$ given by: 
\begin{equation*}
(g,x) \backsim (h,y) \iff \exists n \in N, \,\, y=n \cdot x \text{ and } g^{-1}hn \in \mathcal{U}_{x}(K_P).
\end{equation*}
The Bruhat-Tits tree (building) $\X_P$ defined from $\mathcal{G}_{K_P}$ is the gluing of multiple copies of $\Aa_P$ according to:
$$
\X_P=\X\big(\mathcal{G}_{K_P}\big):= \mathcal{G}_{K_P}(K_P) \times \Aa_P/\backsim.
$$
This topological space is a graph, since the relation $\sim$ preserves the tessellation on $\Aa \cong \mathbb{R}  a^{\vee}$.
Moreover, this graph is a tree (i.e., it is connected and simply connected) according to \cite[Prop. 4.33]{Brown-Buildings}.
The group $\mathcal{G}_{K_P}(K_P)$ acts on $\X_P$ via simplicial maps. 
Moreover, the standard apartment $\Aa_{P}$ is the unique double infinity ray of $\X_{P}$ stabilized by $T=\mathcal{T}_{K_P}(K_{P}) \subseteq \mathcal{G}_{K_P}(K_P)$.

In the sequel, we denote by $\mathscr{R}_{P}$ the ray of $\Aa_{P}$ whose vertex set is $\lbrace v \in \mathrm{V}(\Aa_P): \texttt{a}(v)\geq 0 \rbrace$.
Note that $\Aa_P$ equals $\mathscr{R}_{P} \cup \mathrm{s} \cdot \mathscr{R}_{P}$, where $\mathscr{R}_{P}\cap \mathrm{s} \cdot \mathscr{R}_{P} = \left\lbrace 0 \cdot \texttt{a}^{\vee} \right\rbrace$.
We enumerate the vertex set $\mathrm{V}(\mathscr{R}_{P})$ of $\mathscr{R}_P$ by writing $\mathrm{V}(\mathscr{R}_{P})=\lbrace v_n \rbrace_{n=0}^{\infty}$, where $v_n$ and $v_{n+1}$ are neighbors for all $n \geq 0$.
We analogously enumerate $\mathrm{V}(\mathrm{s} \cdot \mathscr{R}_{P})$ by setting $\mathrm{V}(\mathrm{s} \cdot \mathscr{R}_{P})=\lbrace v_{-n} \rbrace_{n=0}^{\infty}$, where $v_{-n}$ and $v_{-n-1}$ are neighbors.
Let $v \in \X_{P}$ be a vertex. We denote by $\mathscr{V}^1(v)$ the set of all neighboring vertices of $v$.
The $1$-star of $v$ is the (full) subtree of $\X_P$ whose vertex set is exactly $\mathscr{V}^1(v) \cup \lbrace v \rbrace$.

\subsection{On the action of \texorpdfstring{$\Gamma$}{G} on the Bruhat-Tits tree}\label{subsection quot by arith sub}

Recall that $P=\infty$ is a closed point of $\mathbb{P}^1_F$, which is non-split at $L$.
Thus, the Bruhat-Tits building of $\X_{\infty}$ of $\mathcal{G}_{K_{\infty}}$ is the tree defined in \S \ref{subsection BT}.
Moreover, the group $\Gamma=\mathrm{SU}_3(F[t])$ acts on $\X_{\infty}$ as a subgroup of $\mathcal{G}_{K_{\infty}}(K_{\infty})$.

In the sequel, for any pair of sets $S,T \subseteq L$ we write $H(L,K)_{S\times T}=H(L,K) \cap (S \times T)$, while when $S=T$, we just write $H(L,K)_S$ instead of $H(L,K)_{S\times S}$.
Let $q: F^3\to F$, $q(x,y,z)=2xz+y^2$ be the quadratic form on $F$ defined by the restriction of $h$ to $F$, and consider the following subgroups of $\Gamma$:
\begin{align*}
G_0 &:= \Gamma \cap \mathrm{SL}_3(F) = \mathrm{SO}(q)(F) ,\\ \label{eq H_1}
G_n &:= \left \lbrace \U(x,y) \Tilde{a}(\lambda) \big\vert \, (x,y) \in H(L,K)_{F[\sqrt{t}]}, \,\, \nu_{\infty}(y)\geq -n/2, \,\, \lambda\in F^{*} \right \rbrace, \quad n>0, \\
B_0 &:= B \cap G_0=\{\text{upper triangular matrices in }\mathrm{SO}(q)(F)\}.
\end{align*}

In low dimensions, there are some exceptional isomorphisms between certain algebraic groups.
This is the case for $\mathrm{SO}(q)(F)$ and $\mathrm{PGL}_2(F)$ as the next result, due to Dieudonée in \cite{Dieudonnée}, shows.

\begin{Lemma}\cite[Ch. II, \S 9, (3)]{Dieudonnée}\label{lemma SO3 is PSL2}
There exists an isomorphism $\psi: \mathrm{PGL}_2(F) \to \mathrm{SO}(q)(F)$.
\end{Lemma}

\begin{Remark}
For the sake of completeness, we include a sketch of the proof of Lemma \ref{lemma SO3 is PSL2} by describing the action of $\mathrm{PGL}_2$ on the Lie algebra of $\mathrm{SL}_2$.
Indeed, since $\mathrm{char}(F)\neq 2$, up to replacing $q$ by an equivalent quadratic form, we can assume that $q(x,y,z)=xz+y^2$, $\forall x,y,z \in F$.
Since $\mathrm{PGL}_2(F)$ acts, by conjugation, on the set of trace null matrices $V=\left\lbrace A_{a,b,c}=\ssbmatrix{b}{a}{c}{-b} \vert a,b,c \in F\right\rbrace$, we have a representation $\psi: \mathrm{PGL}_2(F) \to \mathrm{GL}(V) \cong \mathrm{GL}_3(F)$.
By definition:
$$\ker(\psi)=\lbrace [B] \in  \mathrm{PGL}_2(F) \vert \, BAB^{-1}=A, \forall A \in V \rbrace.$$
Since $V \oplus F \cdot \mathrm{id} = \mathbb{M}_2(F)$, the kernel of $\psi$ equals:
$$\left\lbrace [B] \in  \mathrm{PGL}_2(F) \vert \, BAB^{-1}=A, \forall A \in \mathbb{M}_2(F) \right\rbrace=\left\lbrace [B] \in  \mathrm{PGL}_2(F) \vert \, B \in F^* \cdot \mathrm{id} \right\rbrace= \left\lbrace [\mathrm{id}] \right\rbrace.$$
Thus $\psi$ is injective.
Since $A_{a,b,c}^2=(ac+b^2) \cdot \mathrm{id}= q((a,b,c)) \cdot \mathrm{id}$ and $(\psi([B])(A))^2 = (BAB^{-1})^2=A^2$, we get $\mathrm{Im}(\psi) \subset \mathrm{O}(q)(F)$.
Moreover, it is straightforward $\mathrm{Im}(\psi) \subset \mathrm{SL}_3(F)$, whence $\mathrm{Im}(\psi) \subset \mathrm{SO}(q)(F)$.
We can prove that $\psi$ is surjective via a dimensional argument.
\end{Remark}

Next result follows from \cite[\S 11.1]{ArenasBravoLoiselLucchini}.

\begin{Lemma}\cite[\S 11.1]{ArenasBravoLoiselLucchini}\label{lemma ABLL action}
For each $n>0$, the group $G_n$ acts on $\mathscr{V}^1(v_n)$ with two orbits, namely $G_n \cdot v_{n-1}$ and $G_n \cdot v_{n+1}$. 
Moreover, the group $G_0$ acts transitively on $\mathscr{V}^1(v_0)$.
In particular, the ray $\mathscr{R}_{\infty}$ is a fundamental domain for the action of $\Gamma=\mathrm{SU}_3(F[t])$ on $\X_{\infty}$.
Moreover, for each $n \geq 0$, the $\Gamma$-stabilizer of $v_n$ is exactly $G_n$.
\end{Lemma}

By applying Bass-Serre Theory (cf. \cite[\S 5]{SerreTrees}) to the preceding result, with Arenas-Carmona, Loisel and Lucchini Arteche, we decompose the group $\Gamma$ in the following amalgamated product.

\begin{Corollary}\cite[Th. 11.1]{ArenasBravoLoiselLucchini}\label{lemma ABLL amal}
The group $\Gamma$ is the free product of $\mathrm{SO}(q)(F)\cong \mathrm{PGL}_2(F)$ and 
$$B:=\left\lbrace \U(x,y)\widetilde{a}(\lambda) \big\vert \, (x,y) \in H(L,K)_{F[\sqrt{t}]}, \, \, \lambda \in F^{*} \right\rbrace,$$
amalgamated along their intersection $B_0$.
\end{Corollary}

\section{The homology of \texorpdfstring{$\mathrm{SU}_3(F[t])$}{SU31}}\label{section homolofy of SU(F[t])}

The main goal of this section is to prove Th. \ref{main teo 1}.
In order to do this, we develop some results that describe the homology of suitable subgroups of $\mathcal{B}_K(K)$.
We will also apply them in order to prove Th. \ref{main teo hom h1} in \S \ref{subsection action of G'}.

Let $S,T$ be two $F$-subvector spaces of $L$ such that $\mathrm{N}(S)=\mathrm{N}_{L/K}(S) \subseteq T$, and write:
$$B_{S,T}:=\Big\lbrace \U(x,y)\tilde{a}(\lambda) \big\vert \, (x,y)\in H(L,K)_{S\times T}, \,\, \lambda \in F^* \Big\rbrace. $$
The first part of this section is devoted to prove the next result:
\begin{Proposition}\label{prop inv hom for Bst}
When $\mathrm{char}(F)=0$, the group homomorphism $F^*\to B_{S,T}$, $\lambda\mapsto \tilde{a}(\lambda)$ induces isomorphisms $H_*(F^*, \mathbb{Z}) \xrightarrow{\cong} H_*(B_{S,T}, \mathbb{Z})$.
\end{Proposition}

Let $U_T^0$ be the subgroup of $B_{S,T}$ defined by $U_T^0= \lbrace \V(y) \vert \, y \in H(L,K)^0 \cap T \rbrace$. It is not hard to prove that $U_T^0$ is isomorphic to the additive group $ T^0:=\lbrace y \in T \vert \, \mathrm{Tr}(y)=0 \rbrace$.

\begin{Lemma}\label{lemma quot B/U0}
The group $U_T^0$ is normal in $B_{S,T}$ and $B_{S,T}/U_T^0 \cong F^* \ltimes S$, where $\lambda \in F^*$ acts on $S$ via $x \mapsto \lambda\cdot x$.
\end{Lemma}

\begin{proof}
Let $U_{S,T}$ be the unipotent radical of $B_{S,T}$. More explictly, let:
$$U_{S,T}:=\left\lbrace \U(x,y) \big \vert \, (x,y)\in H(L,K)_{S\times T} \right\rbrace. $$
Since $U_{S,T} \trianglelefteq B_{S,T}$ with $B_{S,T}/U_{S,T} \cong F^*$, and $F^*$ is a subgroup of $B_{S,T}$, we have that
 $$B_{S,T} \cong U_{S,T} \ltimes F^{*}.$$
Note that, for each $(x,y) \in H(L,K)_{S \times T}$, $\lambda \in F^*$ and $v \in H(L,K)^0_{T}$, we have:
\begin{equation}\label{eq action of B on u2a}
\Big( \U(x,y)\tilde{a}(\lambda) \Big) \cdot \V(v) \cdot \Big( \U(x,y)\tilde{a}(\lambda) \Big)^{-1}=\V(\lambda^2v). 
\end{equation}
In particular $U_T^0 $ is a normal subgroup of both $B_{S,T}$ and $U_{S,T}$.
Moreover, since $U_T^0 \cap F^{*} = \lbrace \mathrm{id} \rbrace$, we have that $B/U_T^{0} \cong (U_{S,T}/U_T^0) \ltimes F^{*}$. 
Let $\pi_S: U_{S,T} \to S$ be the map defined by $\pi_S(\U(x,y))=x$.
Since
$$\U(x,y)\U(u,v) = \U(x+u,y+v-\bar{x}u),$$
the map $\pi_S$ is a group homomorphism.
Moreover, since $\mathrm{N}(S) \subseteq T$, for each $x\in S$, we have that $\U(x,-\mathrm{N}(x)/2) \in H(L,K)_{S\times T}$.
Thus $\pi_S$ is surjective.
By definition of $U_T^0$, we have $\ker(\pi_S)=U_T^{0}$, whence we conclude that $B/U_T^{0} \cong F^{*} \ltimes S$.
It is straightforward that 
$$\Tilde{a}(\lambda) \U(x,y) \Tilde{a}(\lambda) = \U(\lambda x, \lambda^2y).$$
Thus, $F^{*}$ acts on $S$ via $x \mapsto \lambda\cdot x$.
\end{proof}

Since $B_{S,T}/U_T^0$ is a semi-direct product of $F^*$ with a $F$-vector space, namely $S$, where $F^*$ acts via  scalar multiplication,
the next result follows from \cite[Prop. 3.2]{Knudson1}.

\begin{Lemma}\label{lemma inv hom in B/U}
When $\mathrm{char}(F)=0$, the canonical map $F^* \to B_{S,T}/U_T^0$ induces isomorphisms $H_*(F^*,\mathbb{Z}) \xrightarrow{\cong} H_*\big(B_{S,T}/U_T^0,\mathbb{Z}\big)$. \qed
\end{Lemma}

Now, in order to prove Prop. \ref{prop inv hom for Bst}, we check that the canonical homomorphism $B_{S,T} \to B_{S,T}/U_T^{0}$ induces isomorphisms in the associated integral homology groups.
To do so, we first describe an action of $B_{S,T}$ on $U_T^0$.
Indeed, the next result directly follows from Eq. \eqref{eq action of B on u2a}:

\begin{Lemma}\label{lemma action of F* on U0}
The group $B_{S,T}$ acts on $U_T^0$ via $g \cdot v= \lambda^2 v$, where $g=\big( \U (x,y) \tilde{a}(\lambda) \big) \in B_{S,T}$ and $v \in T^0 \cong U_T ^0$.
\end{Lemma}

The next result, together with Lemma \ref{lemma inv hom in B/U}, proves Prop. \ref{prop inv hom for Bst}.

\begin{Lemma}\label{lemma inv hom 2}
The group homomorphism $B_{S,T} \to B_{S,T}/U_T^{0}$ induces isomorphisms in the associated homology groups, i.e., we have $H_*\big(B_{S,T}, \mathbb{Z}\big) \cong H_*\big( B_{S,T}/U_T^{0}, \mathbb{Z}\big)$.
\end{Lemma}

\begin{proof}
Since $U_T^{0}$ is a normal subgroup of $B_{S,T}$ we have the (non-split) exact sequence:
$$ 0 \to U_T^{0} \to B_{S,T} \to B_{S,T}/U_T^{0} \to 1. $$
The Hochschild–Serre spectral sequence associated is then:
$$ E^2_{p,q}= H_p \Big( B_{S,T}/U_T^{0} ,\,\, H_q\big(U_T^{0} , \mathbb{Z}\big) \Big) \Rightarrow H_{p+q}\big(B_{S,T}, \mathbb{Z}\big). $$
Since $U_T^{0} \cong (T^0,+)$ is a torsion-free abelian group, we have that $H_q\big(U_T^{0}, \mathbb{Z}\big) \cong \bigwedge^q\big(T^0\big)$.
It follows from Lemma \ref{lemma quot B/U0} that $F^{*}$ injects into $ B_{S,T}/U_T^{0}$.
Moreover, the action of any $\lambda \in F^{*}$ on $v \in T^{0}$ is given by $\alpha \cdot v = \alpha^2 v$, according to Lemma \ref{lemma action of F* on U0}.
Therefore $\lambda \in F^*$ acts on $H_q\big(U_T^{0}, \mathbb{Z}\big)$ as multiplication by $\lambda^{2q}$.

Let us fix $\lambda_0 \in (F^*)^{2q}$ and write $M:=\bigwedge^q\big( T^0 \big) \cong H_q\big(U_T^{0}, \mathbb{Z}\big)$.
Since $\lambda_0$ belongs to $F^* \hookrightarrow B_{S,T}/U_T^{0}$, it follows from \cite[Prop. 8.1]{Brown-Cohomology} that:
\begin{equation}\label{Eq fix homology}
\lambda_0 \cdot z=z, \quad \forall z \in H_*\big(B_{S,T}/U_T^{0}, M \big). 
\end{equation}
Now, recall that $\mathbb{Q}^* \subseteq F^*$, since $\mathrm{char}(F)=0$.
In particular, we can set $\lambda_0=n^{2q}$, for any $n \in \mathbb{Z}\smallsetminus \lbrace -1,1 \rbrace$.
Thus, it follows from Eq. \eqref{Eq fix homology} that:
$$E^2_{p,q}=H_p \Big(  B_{S,T}/U_T^{0} , \,\, H_q\big(U_T^{0}, \mathbb{Z}\big)\Big),$$
is annihilated by $n^{2q}-1$.
Moreover, since $E^2_{p,q}$ is a $\mathbb{Q}$-vector space, we conclude that $E^2_{p,q}=0$, for all $q>0$.
Thus, $H_p\big(B_{S,T}/U_T^{0} , \mathbb{Z}\big) = E_{p,0}^2 \cong H_p\big(B_{S,T}, \mathbb{Z}\big)$ as desired.
\end{proof}

We now turn to the proof of Theorem~\ref{main teo 1}.
In particular, we now assume that $\mathrm{char}(F)=0$.

\begin{proof}[Proof of Th. \ref{main teo 1}]
The group $\Gamma=\mathrm{SU}_3(F[t])$ is the free product of $\mathrm{SO}(q)(F) \cong \mathrm{PGL}_2(F)$ and $B$ amalgamated along their intersection $B_0$, according to Cor. \ref{lemma ABLL amal}.
This amalgamated product yields a Mayer-Vietoris sequence:
\begin{equation}\label{eq exact sequence in homology with coef. in Z}
\cdots \to H_k\big(B_0, \mathbb{Z}\big) \to H_k\big(\mathrm{PGL}_2(F), \mathbb{Z}\big)\oplus H_k\big(B,\mathbb{Z}\big) \to H_k\big(\Gamma, \mathbb{Z}\big) \to \cdots
\end{equation}
It follows from Prop. \ref{prop inv hom for Bst} applied to $S=T=F$ that the map $F^* \to B_0$, $ \lambda \mapsto \tilde{a}(\lambda),$
induces isomorphisms $H^*(B_0, \mathbb{Z}) \cong H^*(F^*, \mathbb{Z})$.
Analogously, it follows from Prop. \ref{prop inv hom for Bst} applied to $S=T=F[\sqrt{T}]$ that $F^* \to B$, $\lambda \mapsto \tilde{a}(\lambda)$ induces $H_*(B, \mathbb{Z}) \cong H_*( F^{*}, \mathbb{Z})$.
Thus, the exact sequence in \eqref{eq exact sequence in homology with coef. in Z} becomes:
\begin{equation}\label{eq exact sequence in homology with coef. in Z 2}
\cdots \to H_k\big(F^*, \mathbb{Z}\big) \to H_k\big(\mathrm{PGL}_2(F), \mathbb{Z}\big)\oplus H_k\big(F^*,\mathbb{Z}\big) \to H_k\big(\Gamma, \mathbb{Z}\big) \to \cdots
\end{equation}
Note that, since the following diagram commutes,
$$
\def\commutatif{\ar@{}[rd]|{\circlearrowleft}}
\xymatrix{
H_k(B_0, \mathbb{Z}) \ar@{^{(}->}[d] \commutatif & H_k(F^*,\mathbb{Z}) \ar[d]^{\mathrm{id}} \ar[l]_{\cong}
\\
H_k(B,\mathbb{Z}) & \ar[l]_{\cong} H_k(F^*,\mathbb{Z})
 }.$$
the group $H_k\big(F^*, \mathbb{Z}\big)$ maps to $H_k\big(\mathrm{PGL}_2(F), \mathbb{Z}\big)\oplus H_k\big(F^*,\mathbb{Z}\big)$ via the identity map in the second factor.
Thus, Th.~\ref{main teo 1} follows as the long exact sequence in Eq. \eqref{eq exact sequence in homology with coef. in Z 2} breaks up into short exact sequences of the form $0\to H_k\big(F^*, \mathbb{Z}\big) \to H_k\big(\mathrm{PGL}_2(F), \mathbb{Z}\big)\oplus H_k\big(F^*,\mathbb{Z}\big) \to H_k\big(\Gamma, \mathbb{Z}\big) \to 0$.
\end{proof}

\section{\texorpdfstring{$\mathrm{SU}_3(F[t,t^{-1}])$ as an amalgamated product.}{SU32}}\label{section amalgam}

In this section we describe the arithmetic group $\mathrm{SU}_3(F[t,t^{-1}])$ as an amalgamated product of simpler subgroups.
To do so, we do not assume $\mathrm{char}(F)=0$.
Let us keep the notations of \S \ref{subsection BT}.
Recall that $K_0=F(\!(t)\!)$ is the completion of $K=F(t)$ with respect to the valuation $\nu_0$ defined by the closed point $0 \in \mathbb{P}^1_F$.
More concretely, $K_0$ is the completion of $K$ with respect to the valuation $\nu_0: K \to \mathbb{Z} \cup \lbrace \infty \rbrace$ given by $\nu_0(t^n \cdot a/b)=n$, where $a,b\in F[t]$ and $(a,t)=(b,t)=1$.
The integer ring $\mathcal{O}_{K_0}$ of $K_0$ is the ring of formal power series $\mathcal{O}_{K_0}=F[\!\lvert t \rvert\! ]$, where $\pi_{K_0}=t$ is a uniformizing parameter.
The valuation $\nu_0$ on $K$ extends to $L=F(\sqrt{t})$ by setting $\nu_0(\sqrt{t})=1/2$.
The completion of $L$ with respect to $\nu_0$ is $L_0=F(\!(\sqrt{t})\!)$ and its integer ring is $\mathcal{O}_{L_0}=F[\!\lvert \sqrt{t} \rvert \!]$, where $\pi_{L_0}=\sqrt{t}$ is a uniformizing parameter.

The Bruhat-Tits building $\mathcal{X}_0$ of $\mathrm{SL}_{3,L_0}$ is the simplicial complex whose vertex set corresponds to the homothety classes $[\Lambda]$ of $\mathcal{O}_{L_0}$-lattices $\Lambda \subset L_0^3$, where $[\Lambda]$ and $[\Lambda']$ are neighbors exactly when there are representatives $\Lambda_0\in [\Lambda]$ and $\Lambda_0'\in[\Lambda']$ such that $\Lambda_0 \subseteq \Lambda_0'$ and $\Lambda_0'/\Lambda_0 \cong \mathcal{O}_{L_0}/\pi_{L_0} \mathcal{O}_{L_0} \cong F$. 
The group $\mathrm{SL}_3(L_0)$ acts on $V(\mathcal{X}_0)$ via $g \cdot [\Lambda]=[g(\Lambda)]$, for all $g\in \mathrm{SL}_3(L_0)$ and all $[\Lambda] \in V(\mathcal{X}_0)$.
This action is simplicial, so that it extends to the full space $\mathcal{X}_0$.
See \cite[\S 6.9.2]{AbramenkoBrown} for more details.
A (simplicial) chamber of $\mathcal{X}_0$ is a simplex of maximal rank.
For example, the equilateral triangle $\mathcal{C}_0 \subset \mathbb{R}^2$ defined by the vertex set $\lbrace v_0, v_1, v_2 \rbrace$, where:
\begin{equation}
v_0= \left[ \bmattrix{\mathcal{O}_{L_0}}{\mathcal{O}_{L_0}}{\mathcal{O}_{L_0}} \right], \quad v_1= \left[ \bmattrix{\mathcal{O}_{L_0}}{\mathcal{O}_{L_0}}{\pi_{L_0}\mathcal{O}_{L_0}} \right] \text{ and } v_2= \left[\bmattrix{\mathcal{O}_{L_0}}{\pi_{L_0}\mathcal{O}_{L_0}}{\pi_{L_0}\mathcal{O}_{L_0}} \right],
\end{equation}
is a chamber of $\mathcal{X}_0$.
Any other chamber of $\mathcal{X}_0$ is $\mathrm{SL}_3(L_0)$-conjugate to $\mathcal{C}_0$, and any two different faces of $\mathcal{C}_0$ fail to be $\mathrm{SL}_3(L_0)$-conjugates.
In other words, $\mathcal{C}_0$ is a fundamental domain for the action of $\mathrm{SL}_3(L_0)$ over $\mathcal{X}_0$. See \cite[1.69]{AbramenkoBrown} for details.
In the sequel, we denote by $e$ the edge of $\mathcal{C}_0$ connecting $v_1$ with $v_2$. See Figure \ref{Figure building SL3}(A).
The Bruhat-Tits tree $\X_0$ of $\mathcal{G}_{K_0}$ injects into the Bruhat-Tits building $\mathcal{X}_0$ of $\mathrm{SL}_{3,L_0}$ in a way that:
\begin{itemize}
    \item the middle point $v_1^{*}$ of $e$ is a vertex of $\X_0$, and
    \item one edge of $\X_0$ consists in the segment $e^{*}$ joining $v_0$ with $v_1^{*}$.
\end{itemize}
Any other edge of $\X_0$ is $\mathcal{G}_{K_0}(K_0)$-conjugate to $e^{*}$, and the vertices $v_0$ and $v_1^{*}$ fail to be  $\mathcal{G}_{K_0}(K_0)$-conjugates.
This construction is a consequence of a more general result due to Prasad and Yu in \cite{Prasad-Yu}.

Let $\tilde{\Gamma}=\mathrm{SU}_3\big(F[t,t^{-1}]\big)$ be the group of $F[t,t^{-1}]$-point of $\mathcal{G}=\mathrm{SU}_{3, \mathbb{P}^1_F}$.
This group corresponds to the subgroup of matrices in $\mathrm{SL}_3\big(F[\sqrt{t},1/\sqrt{t}]\big)$ preserving the hermitian form $h_{F[t,t^{-1}]}$ defined in Eq. \eqref{eq h_R}.

\begin{Lemma}\label{lemma fund domain for H}
The edge $e$ is a fundamental domain for the action of $\tilde{\Gamma}$ on $\X_0$.
\end{Lemma}

\begin{proof}
On one hand, since $e$ is a chamber (edge) of $\X_0$, it follows from \cite[1.69]{AbramenkoBrown} that $e$ is a fundamental domain for the action of $\mathcal{G}_{K_0}(K_0)$ on $\X_0$.
On the other hand, Strong approximation theorem over $\mathrm{SU}_3$ (cf.~\cite[Thm.~A]{PrasadSA}) implies that $\tilde{\Gamma}$ is dense in $\mathcal{G}_{3,K_0}(K_0)$.
Thus, since the $\mathcal{G}_{3,K_0}(K_0)$-stabilizers of any simplex in $\X_0$ is open, we have that $e$ is a fundamental domain for the action of $\tilde{\Gamma}$ on $\X_0$.
\end{proof}
Let us introduce the following subgroups of $\tilde{\Gamma}:$
\begin{align}\label{eq H_0}
\hat{\Gamma} &:= \mscriptsize{ \begin{pmatrix} F[\sqrt{t}] & F[\sqrt{t}] & \sqrt{t}^{-1}F[\sqrt{t}] \\ \sqrt{t}F[\sqrt{t}] & F[\sqrt{t}] & F[\sqrt{t}] \\ \sqrt{t}F[\sqrt{t}] & \sqrt{t}F[\sqrt{t}] & F[\sqrt{t}] \end{pmatrix}}\cap \tilde{\Gamma},\\  \label{eq H_{0,1}}
\Gamma_{0}&:=\Gamma \cap \hat{\Gamma}= \mscriptsize{ \begin{pmatrix} F[\sqrt{t}] & F[\sqrt{t}] & F[\sqrt{t}] \\ \sqrt{t}F[\sqrt{t}] & F[\sqrt{t}] & F[\sqrt{t}] \\ \sqrt{t}F[\sqrt{t}] & \sqrt{t}F[\sqrt{t}] & F[\sqrt{t}] \end{pmatrix}}\cap \tilde{\Gamma}.
\end{align}
The group $\Gamma_{0}=\Gamma_{0}(t)$ is the Hecke congruence subgroup of $\Gamma$, while $\hat{\Gamma}$ is an arithmetic subgroup of $\tilde{\Gamma}$, which is non-isomorphic to $\Gamma$.
Next result describes $\tilde{\Gamma}$ as an amalgamated product of the preceding subgroups.

\begin{figure}
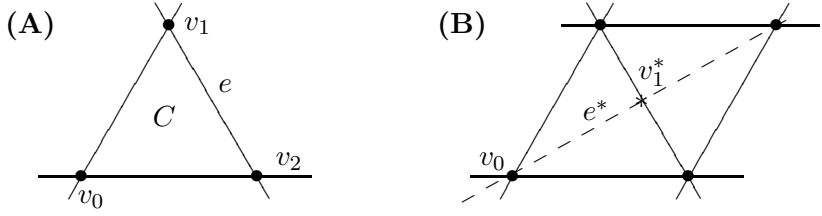

$$
\xygraph{
!{<0cm,0cm>;<1.0cm,0cm>:<0cm,1.0cm>::}
!{(2.8,2) }*+{\textbf{(A)}}="A"
!{(2.8,0) }*+{}="a2"
!{(6.6,0) }*+{}="b2"
!{(3.24,-0.4) }*+{}="c2"
!{(5.56,-0.4) }*+{}="c3"
!{(4.851,2.4) }*+{}="d2"
!{(2.079,2.4) }*+{}="e2"
!{(4.381,2.4) }*+{}="e3"
!{(6,-0.4) }*+{}="f3"
!{(2.7,-0.4) }*+{}="g1"
!{(7.161,2.133) }*+{}="g2"
!{(3.465,0) }*+{\bullet}="h2"
!{(3.6,-0.3) }*+{v_0}="h2n"
!{(5.36,1.2) }*+{e}="j2"
!{(4.62,2) }*+{\bullet}="k3"
!{(5.775,0) }*+{\bullet}="k4"
!{(5,2.0) }*+{v_1}="k3n"
!{(6.2,0.2) }*+{v_2}="k4n"
!{(4.56,0.8) }*+{C}="j3"
"a2"-"b2"
"c2"-"d2" 
"e3"-"f3"
} \quad
\xygraph{
!{<0cm,0cm>;<1.0cm,0cm>:<0cm,1.0cm>::}
!{(2.8,2) }*+{\textbf{(B)}}="B"
!{(2.8,0) }*+{}="a2"
!{(4,2) }*+{}="a3"
!{(6.6,0) }*+{}="b2"
!{(7.6,2) }*+{}="b3"
!{(3.24,-0.4) }*+{}="c2"
!{(5.56,-0.4) }*+{}="c3"
!{(4.851,2.4) }*+{}="d2"
!{(7.161,2.4) }*+{}="d3"
!{(2.079,2.4) }*+{}="e2"
!{(4.381,2.4) }*+{}="e3"
!{(6,-0.4) }*+{}="f3"
!{(2.7,-0.4) }*+{}="g1"
!{(7.161,2.133) }*+{}="g2"
!{(3.465,0) }*+{\bullet}="h2"
!{(6.93,2) }*+{\bullet}="h3"
!{(3.2,0.2) }*+{v_0}="h2n"
!{(5.16,1) }*+{\ast}="j2"
!{(5.3,1.38) }*+{v_1^{*}}="j2"
!{(4.62,2) }*+{\bullet}="k3"
!{(5.775,0) }*+{\bullet}="k4"
!{(4.56,0.88) }*+{e^{*}}="j3"
"a2"-"b2" "a3"-"b3"
"c2"-"d2" "c3"-"d3"
"e3"-"f3"
"g1"-@{--}"g2"
}
$$
\caption{In the left side, the fundamental chamber of the building of $\mathrm{SL}_3$. 
In the right side, the tree of $\mathrm{SU}_3$ insides the building for $\mathrm{SL}_3$.}\label{Figure building SL3}
\end{figure}

\begin{Lemma}\label{lemma decomposition of H}
The group $\tilde{\Gamma}$ is isomorphic to $\Gamma \ast_{\Gamma_0} \hat{\Gamma}$.
\end{Lemma}

\begin{proof}
By applying the Bass-Serre theory (cf. \cite[\S 6, Theo. 6]{SerreTrees}) on the fundamental domain described in Lemma \ref{lemma fund domain for H}, we get that $\tilde{\Gamma}$ is isomorphic to the free product of $\mathrm{Stab}_{\tilde{\Gamma}}(v_0)$ with $\mathrm{Stab}_{\tilde{\Gamma}}(v_1^{*})$ amalgamated by $\mathrm{Stab}_{\tilde{\Gamma}}(e^{*})$.
Thus, we just need to compute the preceding groups.
Indeed, it is straightforward that $\mathrm{Stab}_{\mathrm{SL}_3(L_0)}(v_0)=\mathrm{SL}_3(\mathcal{O}_{L_0})$.
Moreover, since the $\mathrm{SL}_3(L_0)$-stabilizer of $e$ equals $\mathrm{Stab}_{\mathrm{SL}_3(L_0)}(v_1) \cap \mathrm{Stab}_{\mathrm{SL}_3(L_0)}(v_2)$, we have:
\begin{equation*}
   \mathrm{Stab}_{\mathrm{SL}_3(L_0)}(e)= \mscriptsize{ \begin{pmatrix} \mathcal{O}_{L_0} & \mathcal{O}_{L_0} & \pi_{L_0}^{-1}\mathcal{O}_{L_0} \\ \pi_{L_0}\mathcal{O}_{L_0} & \mathcal{O}_{L_0} & \mathcal{O}_{L_0} \\ \pi_{L_0} \mathcal{O}_{L_0} & \pi_{L_0} \mathcal{O}_{L_0} & \mathcal{O}_{L_0}\end{pmatrix}} \cap \mathrm{SL}_3(L_0).
\end{equation*}
By definition of stabilizers we have that $\mathrm{Stab}_{\tilde{\Gamma}}(v_0)=\mathrm{Stab}_{\mathrm{SL}_3(L_0)}(v_0)\cap \tilde{\Gamma}$.
Since $v_1^*$ is the middle point of $e$, we get that $\mathrm{Stab}_{\tilde{\Gamma}}(v_1^{*})=\mathrm{Stab}_{\mathrm{SL}_3(L_0)}(e)\cap \tilde{\Gamma}$.
Since:
\begin{align*}
F[\sqrt{t}]&=\mathcal{O}_{L_0} \cap F[\sqrt{t},1/\sqrt{t}], \\
\sqrt{t}F[\sqrt{t}]&=\pi_{L_0}\mathcal{O}_{L_0} \cap F[\sqrt{t},1/\sqrt{t}], \\
(1/\sqrt{t})F[\sqrt{t}]&=\pi_{L_0}^{-1}\mathcal{O}_{L_0} \cap F[\sqrt{t},1/\sqrt{t}],
\end{align*}
we conclude that $\Gamma=\mathrm{Stab}_{\tilde{\Gamma}}(v_0)$ and $\hat{\Gamma}=\mathrm{Stab}_{\tilde{\Gamma}}(v_1^{*})$.
Moreover, since $\Gamma_0=\Gamma \cap \hat{\Gamma}$, we obtain that $\Gamma_0=\mathrm{Stab}_{\tilde{\Gamma}}(e^*)$, whence the result follows.
\end{proof}

\subsection{On the action of \texorpdfstring{$\hat{\Gamma}$}{G} on the Bruhat-Tits tres}\label{subsection action of G'}

In this section, we focus on the description of the arithmetic group $\hat{\Gamma}$ via its action on the Bruhat-Tits tree $\X_{\infty}$.
Let us recall that $K_{\infty}=F(\!(1/t)\!)$
is the completion of $K=F(t)$ with respect to the valuation $\nu_{\infty}: K \to \mathbb{Z} \cup \lbrace \infty \rbrace$ given by $\nu_{\infty}(a/b)=\deg(b)-\deg(a)$, where $a,b\in F[t]$ and $(a,b)=1$.
The integer ring $\mathcal{O}_{K_{\infty}}$ of $K_{\infty}$ is the ring of Laurent series $F[\!\lvert 1/t \rvert\!]$.
In particular $\pi_{K_{\infty}}=1/t$ is a uniformizing parameter in $\mathcal{O}_{K_{\infty}}$.
The quadratic extension $L_{\infty}$ of $K_{\infty}$ is $F(\!(1/\sqrt{t})\!)$, where $\mathcal{O}_{L_{\infty}}=F[\!\lvert 1/\sqrt{t}\rvert\!]$.
Then, a uniformizing parameter in $\mathcal{O}_{K_\infty}$ is $\pi_{L_{\infty}}=1/\sqrt{t}$.

In the sequel, we denote by $S$ the ring $S=F[\sqrt{t}]$ and by $J$ its principal ideal $J=\pi_{L_{\infty}}^{-1}S=\sqrt{t} F[\sqrt{t}]$.
In particular, $J^{-1}=\pi_{L_{\infty}}S=(1/\sqrt{t}) F[\sqrt{t}]$ and
$$\hat{\Gamma} = \mscriptsize{ \begin{pmatrix} S & S & J^{-1} \\ J & S & S \\ J & J & S \end{pmatrix}}\cap \tilde{\Gamma}. $$
As in \S \ref{subsection BT}, here we enumerate the vertices in $\Aa_{\infty}$ by writing $V(\Aa_{\infty})=\lbrace v_n \rbrace_{n=-\infty}^\infty$, where $v_i$ and $v_{i+1}$ are neighbors and $v_0$ is the unique vertex in $\Aa_{\infty}$ satisfying $\texttt{a}(v_0)=0$.
In order to describe a fundamental domain for the action of $\hat{\Gamma}$ on $\X_{\infty}$, we start by describing the $\hat{\Gamma}$-stabilizers of some suitable vertices $v \in \Aa_{\infty}$.

\begin{Lemma}
For each $n \geq 0$, the group $\mathrm{Stab}_{\hat{\Gamma}}(v_n)$ equals
$$\hat{G}_{n}:=\left \lbrace\U(x,y) \Tilde{a}(\lambda) \big\vert \, (x,y) \in H(L,K)_{S\times J^{-1}}, \,\, \nu_{\infty}(y)\geq -n/2, \,\, \lambda\in F^{*} \right \rbrace,$$
while for $n=-1$, the group $\mathrm{Stab}_{\hat{\Gamma}}(v_n)$ equals: 
$$\hat{G}_{-1}:= \left \lbrace \sbmattrix{a}{0}{\sqrt{t}^{-1}b}{0}{1}{0}{\sqrt{t}c}{0}{d} \Bigg{\vert} \, a,b,c,d \in F, ad-bc=1 \right \rbrace \cong \mathrm{SL}_2(F).$$
\end{Lemma}

\begin{proof} We divide this proof in two parts, according if $n\geq 0$ or $n=-1$.
Firstly, assume that $n\geq 0$. 
By definition, we have $\mathrm{Stab}_{\hat{\Gamma}}(v_n)=\hat{\Gamma} \cap \mathrm{Stab}_{\mathcal{G}_{K_{\infty}}(K_{\infty})}(v_n)$.
As in \S \ref{subsection BT}, we denote by $N$ the group of $K_{\infty}$-points of $\mathcal{N}_{\mathcal{G}_{K_{\infty}}}(\mathcal{S}_{K_{\infty}})$.
The bounded torus $T_b=\mathcal{T}_{K_{\infty}}(K_{\infty})_{\mathrm{bound}}$ is the set of diagonal matrices of the form $\tilde{a}(\lambda)$, where $\lambda \in \mathcal{O}_{L_{\infty}}^{*}$.
It follows from \cite[9.3(i) and 8.10 (ii)]{Landvogt} that $\mathrm{Stab}_{\mathcal{G}_{K_{\infty}}(K_{\infty})}(v_n)$ decomposes as $U_{-a,n} U_{a,n} N_{n}$, where:
\begin{align*}\label{eqn def U_a_lambda}
\begin{split}
	U_{-\texttt{a},n}
	&= \big\lbrace
		\mathrm{u}_{-\texttt{a}}(u,v) \vert \, (u,v) \in H(L_{\infty}, K_{\infty}), \,
		\nu_{\infty}(v) \geq n/2
	\big \rbrace,\\
	U_{\texttt{a},n}
	&= \big \lbrace
	\U(x,y) \vert \, (x,y) \in
        H(L_{\infty}, K_{\infty}), \,
  \nu_{\infty}(y) \geq -n/2
	\big \rbrace, \\
	N_{n}
	&= \mathrm{Stab}_{N}(v_n)
	= \{1, \widetilde{a}(r) \mathrm{s}\} \cdot T_b,
 \end{split}
\end{align*}
with $r\in L_{\infty}^*$ and $\nu_{\infty}(r) = -n/2$ according to \cite[8.6(ii), 4.21(iii) and 4.14(i)]{Landvogt}.
Thus, it is clear that:
$$\hat{G}_{n} \subseteq U_{\texttt{a},n} T_b \cap \hat{\Gamma} \subseteq  \mathrm{Stab}_{\hat{\Gamma}}(v_n).$$
In order to prove the inverse contention, let $g \in \mathrm{Stab}_{\hat{\Gamma}}(v_n)$ and write $g=\mathrm{u}_{-\texttt{a}}(u,v) \U(x,y) \mathrm{m}$, with $\mathrm{u}_{-\texttt{a}}(u,v) \in U_{-\texttt{a},n}$, $\U(x,y) \in U_{\texttt{a},n}$ and $\mathrm{m} \in N_n$.
First, assume that $\mathrm{m} \in T_b$, so that $\mathrm{m}=\tilde{a}(\lambda)$, with $\lambda\in \mathcal{O}_{L_{\infty}}^{*}$.
Then:
\begin{equation*}
    g = \sbmattrix{1}{0}{0}{u}{1}{0}{v}{-\bar{u}}{1}
    \sbmattrix{1}{-\bar{x}}{y}{0}{1}{x}{0}{0}{1}
    \sbmattrix{\lambda}{0}{0}{0}{\bar{\lambda}/\lambda}{0}{0}{0}{1/\bar{\lambda}}
    =\sbmattrix{\lambda}{-(\bar{\lambda}/\lambda)\bar{x}}{(1/\bar{\lambda})y}{\lambda u}{(\bar{\lambda}/\lambda)(1-\bar{x}u)}{(1/\bar{\lambda})(x+uy)}{\lambda v}{-(\bar{\lambda}/\lambda)(\bar{u}+v\bar{x})}{(1/\bar{\lambda})(1-\bar{u}x+vy)}.
\end{equation*}
This implies that $\lambda \in S \cap \mathcal{O}_{L_{\infty}}^{*}=F^{*}$. Since $\nu_{\infty}( J \smallsetminus \lbrace 0 \rbrace) \leq -1/2$ and $\nu_{\infty}(v) \geq n/2\geq 0$, we obtain that $v=0$. Since $\mathrm{N}(u)+\mathrm{Tr}(v)=0$, we also have $u=0$.
This proves that $g\in \hat{G}_{n}$ when $\mathrm{m} \in T_b$.

Now, assume that $\mathrm{m} \in (\widetilde{a}(r) \mathrm{s} )\cdot T_b$, where $\nu_{\infty}(r)= -n/2$.
Since $\mathrm{s} \cdot \tilde{a}(\mu)= \tilde{a}(\bar{\mu}^{-1}) \cdot \mathrm{s}$, for all $\mu \in L_{\infty}^{*}$, we can write
$\mathrm{m} = \tilde{a}(\lambda) \mathrm{s}$, for some $\lambda\in L_{\infty}^{*}$ with $\nu_{\infty}(\lambda) \geq  \nu_{\infty}(r) = -n/2$.
Then:
\begin{equation*}
    g =\sbmattrix{-(1/\bar{\lambda}) y}{(\bar{\lambda}/\lambda)\bar{x}}{-\lambda}{-(1/\bar{\lambda})(x+uy)}{-(\bar{\lambda}/\lambda)(1-\bar{x}u) }{-\lambda u}{-(1/\bar{\lambda})(1-\bar{u}x+vy)}{(\bar{\lambda}/\lambda)(\bar{u}+v\bar{x})}{-\lambda v} \in \hat{\Gamma}.
\end{equation*}
Note that $\lambda v \in S$ and $(1/\bar{\lambda})y \in S$.
Then, we have $\nu_{\infty}(\lambda v), \nu_{\infty}((1/\bar{\lambda})y) \leq 0$.
Moreover, since $\nu_{\infty}(\lambda)\geq -n/2$ and $\nu_{\infty}(v) \geq n/2$, we have that $\nu_{\infty}(\lambda v)\geq 0$.
Thus $\nu_{\infty}(\lambda)=-n/2$ and $\nu_{\infty}(v)=n/2$.
By an analogous argument, applied to $(1/\bar{\lambda})y \in S$, we get that $\nu_{\infty}(y)=-n/2$.
It follows from the relations $\mathrm{N}(u)=-\mathrm{Tr}(v)$ and $\mathrm{N}(x)=-\mathrm{Tr}(y)$ that $\nu_{\infty}(u) \geq n/4$ and $\nu_{\infty}(x)\geq -n/4$.
This implies the following:
$$\nu_{\infty}\big((1/\bar{\lambda})(x+uy)\big) \geq 0, \quad \nu_{\infty}\big((\bar{\lambda}/\lambda)(\bar{u}+v\bar{x})\big) \geq 0, \text{ and } \nu\big((1/\bar{\lambda})(1-\bar{u}x+vy)\big) \geq 0.$$
However $(1/\bar{\lambda})(x+uy), (\bar{\lambda}/\lambda)(\bar{u}+v\bar{x})$ and $ (1/\bar{\lambda})(1-\bar{u}x+vy)$ belong to $J$, which implies that these element are $0$.
Thus $g$ belongs to the Borel subgroup $\mathcal{B}(K)$ of $\mathcal{G}(K)$. 
In other words $g=\U(x_1, y_1) \Tilde{a}(\lambda_1)$, for some $(x_1,y_1) \in H(L,K)$ and some $\lambda_1 \in L^*$.
Since $\lambda_1=-(1/\bar{\lambda}) y \in S \cap \mathcal{O}_{L_\infty}$, we deduce that $\lambda_1\in F^*$.
Moreover, since $x_1=-\lambda u \in S$ and $y_1=-\lambda \in J^{-1}$, we obtain that $(x_1,y_1) \in H(L,K)_{S \times J^{-1}}$.
Finally, since $\nu(y_1)\geq -n/2$, we conclude that $g\in \hat{G}_n$, in either case.

Now, we prove that $\mathrm{Stab}_{\hat{\Gamma}}(v_{-1})=\hat{G}_{-1}$.
Indeed, we can write $\mathrm{Stab}_{\hat{\Gamma}}(v_{-1})$ as $\hat{\Gamma} \cap \mathrm{Stab}_{\mathrm{SL}_3(L_{\infty})}(v_{-1})$, where:
$$\mathrm{Stab}_{\mathrm{SL}_3(L_{\infty})}(v_{-1})= \sbmattrix{ \mathcal{O}_{L_{\infty}} }{ \pi_{L_{\infty}}\mathcal{O}_{L_{\infty}} }{ \pi_{L_{\infty}} \mathcal{O}_{L_{\infty}} }{ \mathcal{O}_{L_{\infty}} }{ \mathcal{O}_{L_{\infty}} }{ \pi_{L_{\infty}}\mathcal{O}_{L_{\infty}} }{ \pi_{L_{\infty}}^{-1} \mathcal{O}_{L_{\infty}} }{\mathcal{O}_{L_{\infty}} }{ \mathcal{O}_{L_{\infty}} } \cap \mathrm{SL}_3(L_{\infty}).$$
Thus, we have:
$$\mathrm{Stab}_{\hat{\Gamma}}(v_{-1})= \sbmattrix{ \mathcal{O}_{L_{\infty}} }{ \pi_{L_{\infty}}\mathcal{O}_{L_{\infty}} }{ \pi_{L_{\infty}} \mathcal{O}_{L_{\infty}} }{ \mathcal{O}_{L_{\infty}} }{ \mathcal{O}_{L_{\infty}} }{ \pi_{L_{\infty}}\mathcal{O}_{L_{\infty}} }{ \pi_{L_{\infty}}^{-1} \mathcal{O}_{L_{\infty}} }{\mathcal{O}_{L_{\infty}} }{ \mathcal{O}_{L_{\infty}} }  \cap \sbmattrix{S}{S}{J^{-1}}{J}{S}{S}{J}{J}{S} \cap \mathrm{SL}_3(L_{\infty}).$$
Since $\mathcal{O}_{L_{\infty}} \cap S=F$, we have that $J \cap \pi_{L_{\infty}}^{-1}\mathcal{O}_{L_{\infty}}= \sqrt{t}F$ and $J^{-1}\cap \pi_{L_{\infty}}\mathcal{O}_{L_{\infty}}=(1/\sqrt{t})F$.
Moreover, since $\pi_{L_{\infty}}\mathcal{O}_{L_{\infty}} \cap S=\mathcal{O}_{L_{\infty}} \cap J=\lbrace 0 \rbrace$, the group $\mathrm{Stab}_{\hat{\Gamma}}(v_{-1})$ equals:
\begin{equation}\label{eq of psi}
\left\lbrace g=\sbmattrix{a_{11}}{0}{(1/\sqrt{t}) a_{13}}{0}{a_{22}}{0}{\sqrt{t}a_{31}}{0}{a_{33}} \Bigg{\vert} \,
\begin{array}{@{}lll@{}}
    g \Phi \bar{g}^{t}=\Phi, \\
    \det(g)=1, \\
     a_{ij}\in F.
  \end{array}
\right\rbrace, \text{ where } \Phi=\sbmattrix{0}{0}{1}{0}{1}{0}{1}{0}{0}.
\end{equation}
More explicitly, an element $g$ as above belongs to $\mathrm{Stab}_{\hat{\Gamma}}(v_{-1})$ if and only if $\det(g)=1$ and:
$$ \sbmattrix{0}{0}{a_{11} a_{33}-a_{13}a_{31} }{0}{a_{22}^2}{0}{a_{11}a_{33}-a_{13}a_{31}}{0}{0}=\sbmattrix{0}{0}{1}{0}{1}{0}{1}{0}{0}.$$
Thus $a_11a_{33}-a_{13}a_{31}=1$ and $a_{22} \in \lbrace \pm 1 \rbrace$.
Since $1=\det(g)=a_{22}(a_11a_{33}-a_{13}a_{31})$, we conclude that $a_{22}=1$, which proves that $\mathrm{Stab}_{\hat{\Gamma}}(v_{-1})=\hat{G}_{-1}$.
\end{proof}

As in Corollary \ref{lemma ABLL amal}, let $B_0$ be the (Borel) subgroup of upper triangular matrices in $G_0=\mathrm{SO}(q)(F)$, which can be written as:
\begin{equation}\label{eq B0}
B_0=\left\lbrace \U(x,-x^2/2) \tilde{a}(\lambda) \vert \, x\in F, \,\,\lambda\in F^*\right\rbrace.
\end{equation}
\begin{Lemma}\label{lemma action in v0}
The group $B_0$ acts on $\mathscr{V}^1(v_0)$ with exactly two orbits, namely $B_0 \cdot v_{1}$ and $B_0 \cdot v_{-1}$.
\end{Lemma}

\begin{proof}
On one hand, it follows from Lemma \ref{lemma ABLL action} that $G_0=\mathrm{SO}(q)(F)$ acts transitively on $\mathscr{V}^1(v_0)$.
On the other hand, note that $B_0$ equals $\mathrm{Stab}_{G_0}(v_1)$.
Then, the orbit set $B_0 \backslash \mathscr{V}^1(v_0)$ is in bijection with the double quotient $B_0 \backslash G_0/B_0$.
Moreover, it follows from the Bruhat decomposition on $G_0$ that this double quotient $B_0 \backslash G_0/B_0$ is in bijection with $\lbrace \mathrm{id}, \mathrm{s} \rbrace$.
Therefore, since $\mathrm{s} \cdot v_1=v_{-1}$, the result follows.
\end{proof}

\begin{Lemma}\label{lemma trans action on v_{-1}}
The group $\hat{G}_{-1}$ acts transitively on $\mathscr{V}^1(v_{-1})$.
\end{Lemma}

\begin{proof}
It follows from Lemma \ref{lemma ABLL action} that $G_1$ acts on $\mathscr{V}^1(v_1)$ with exactly two orbits, namely $G_1 \cdot v_0$ and $G_1 \cdot v_2$.
Since the diagonal group $\lbrace \tilde{a}(\lambda) \vert \lambda \in F^{*} \rbrace$ acts trivially on $ v_0, v_1$ and $v_2$, we have that $U_1:=\lbrace \U(x,y) \vert (x,y) \in H(L,K)_{S}, \nu(y) \geq -1/2\rbrace$ acts on $\mathscr{V}^1(v_1)$ with exactly two orbits, namely $U_1 \cdot v_0$ and $U_1 \cdot v_2$.
Since $\mathrm{s}\cdot v_1=v_{-1}$, $\mathrm{s}\cdot v_0=v_{0}$ and $\mathrm{s}\cdot v_2=v_{-2}$, the group:
$$U_1^{\circ}:=\mathrm{s} U_1 \mathrm{s}^{-1}=\lbrace \mathrm{u}_{-\mathtt{a}}(x,y) \vert (x,y) \in H(L,K)_{S}, \nu(y) \geq -1/2\rbrace,$$
acts on $\mathscr{V}^1(v_{-1})$ with two orbits that are $U_1^{\circ} \cdot v_0$ and $U_1^{\circ} \cdot v_{-2}$.
Note that $y \in S$ satisfies $\nu(y) \geq -1/2$ exactly when $y=a_0+a_1\sqrt{t}$, for certain $a_0,a_1 \in F$. Thus:
$$U_1^{\circ}:=\lbrace \mathrm{u}_{-\mathtt{a}}(x_0,a_0+a_1\sqrt{t}) \vert x_0, a_0, a_1 \in F, x_0^2+2a_0=0\rbrace.$$
Therefore $U_1^{\circ}/ (U_1^{\circ} \cap \mathrm{s}G_0\mathrm{s}^{-1}) = U_1^{\circ}/ (U_1^{\circ} \cap G_0) \cong \lbrace \mathrm{u}_{-\mathtt{a}}(0,y_0\sqrt{t}) \vert  y_0 \in F\rbrace$.
In particular, the group $U_1^{\circ}/ (U_1^{\circ} \cap \mathrm{s}G_0\mathrm{s}^{-1})$ is covered by $\hat{G}_{-1}$, whence $\hat{G}_{-1} \cdot v_0 \supseteq U_1^{\circ} \cdot v_0$.
Analogously, since $U_1^{\circ}/ (U_1^{\circ} \cap \mathrm{s}G_2\mathrm{s}^{-1}) \cong \lbrace \mathrm{id}\rbrace$, we have $\hat{G}_{-1} \cdot v_{-2} \supseteq U_1^{\circ} \cdot v_{-2}$.
Therefore $\hat{G}_{-1}$ acts on $\mathscr{V}^1(v_{-1})$ with at most two orbits, namely $\hat{G}_{-1} \cdot v_0$ and $\hat{G}_{-1} \cdot v_{-2}$. Finally, since the matrix:
$$s_J:=\sbmattrix{0}{0}{\sqrt{t}^{-1}}{0}{1}{0}{\sqrt{t}}{0}{0} \in \hat{G}_{-1},$$
exchanges the vertices $v_0$ with $v_{-2}$, the result follows.
\end{proof}

In the sequel, we denote by $\mathscr{R}_{\infty,-1}$ the ray of $\Aa_{\infty}$ whose vertex set is exactly $\lbrace v_n\rbrace_{n=-1}^{\infty}$.

\begin{Proposition}\label{prop fund dom for H1}
The ray $\mathscr{R}_{\infty,-1}$ is a fundamental domain for the action of $\hat{\Gamma}$ on $\X_{\infty}.$
\end{Proposition}

\begin{proof}
As in \S \ref{subsection quot by arith sub}, let $G_n$ be the $\Gamma$-stabilizer of $v_n$, for $n\geq 0$.
Since, for each $n>0$, the group $\hat{G}_{n}$ contains $G_n$, it follows from Lemma \ref{lemma ABLL action} that $\hat{G}_{n}$ acts on $\mathscr{V}^1(v_n)$ with at most two orbits, namely $\hat{G}_{n}\cdot v_{n-1}$ and $\hat{G}_{n}\cdot v_{n+1}$.
Moreover, since $B_0 \subseteq \hat{G}_{0}$, it follows from Lemma \ref{lemma action in v0} that $\hat{G}_{0}$ acts on $\mathscr{V}^1(v_0)$ with at most two orbits, which are $\hat{G}_{0} \cdot v_{1}$ and $\hat{G}_{0} \cdot v_{-1}$.
Thus, for each vertex $v_n\in \mathscr{R}_{\infty,-1}$, $n\neq 0$, its $\hat{\Gamma}$-stabilizer $\hat{G}_n=\hat{\Gamma}_{v_n}$ acts on $\mathscr{V}^1(v_n)$ with at most two orbits, namely $\hat{G}_{n}\cdot v_{n-1}$ and $\hat{G}_{n}\cdot v_{n+1}$.
Finally, since $\hat{G}_{-1}$ acts transitively on $\mathscr{V}^1(v_{-1})$ according to Lemma \ref{lemma trans action on v_{-1}}, each vertex of $\X_{\infty}$ is in the same $\hat{\Gamma}$-orbit of some vertex in $\mathscr{R}_{\infty,-1}$.
Since the action of $\hat{\Gamma}$ on $\X_{\infty}$ is simplicial, an analogous statement holds for edges.
We conclude that $\mathscr{R}_{\infty,-1}$ contains a fundamental domain for the action of $\hat{\Gamma}$ on $\X_{\infty}.$

Now, we have to prove that any two different vertices in $\mathscr{R}_{\infty,-1}$ fail to belong to the same $\hat{\Gamma}$-orbit.
Indeed, note that, when $v=g \cdot w$, with $g\in \hat{\Gamma}$, then $\mathrm{Stab}_{\hat{\Gamma}}(v)= g \cdot \mathrm{Stab}_{\hat{\Gamma}}(w) \cdot g^{-1}$.
Since $\mathrm{Stab}_{\hat{\Gamma}}(v_{-1})\cong \mathrm{SL}_2(F)$ is non-isomorphic to any $\hat{G}_{n}$, for $n\geq 0$, the vertex $v_{-1}$ is non $\hat{\Gamma}$-equivalent to any $v_n$, with $n\geq 0$.
Assume that $v_n$ is $\hat{\Gamma}$-equivalent to $v_m$, with $n,m\geq 0$.
Then, the groups $\hat{G}_{n}$ and $\hat{G}_{m}$ are $\hat{\Gamma}$-conjugates. 
Following the same argument as in \cite[Lemma 5.5]{ArenasBravoLoiselLucchini} we get that
\begin{align*}
\hat{U}_{n}&:=\left \lbrace\U(x,y) \vert \, (x,y) \in H(L,K)_{S\times J^{-1}}, \,\, \nu_{\infty}(y)\geq -n/2 \right \rbrace, \\
\hat{U}_{m}&:=\left \lbrace\U(x,y) \vert \, (x,y) \in H(L,K)_{S\times J^{-1}}, \,\, \nu_{\infty}(y)\geq -m/2 \right \rbrace,
\end{align*}
are conjugates by a matrix of the form $h=\tilde{a}(\lambda) \U(z,w)$, with $\lambda\in L^*$ and $(z,w)\in H(L,K)$.
Let $\pi: \mathcal{U}_a(K) \to L$ be the group homomorphism defined by $\pi(\U(x,y))=x$.
Then 
$\pi(\hat{U}_{n})=\lbrace s\in S: \nu_{\infty}(s) \geq -n/4 \rbrace.$
Since $h\mathrm{u}_{a}(x,y) h^{-1}= \mathrm{u}_{a}\big((\bar{\lambda}^2/\lambda) x, \lambda\bar{\lambda}(y+\overline{x}z-\overline{z}x)\big)$, we have that $\pi(\hat{U}_{n})= \kappa \cdot \pi(\hat{U}_{m})$, where $\kappa=\overline{\lambda}^2/\lambda$.
In particular, $\pi(\hat{U}_{n})$ and $\pi(\hat{U}_{m})$ are $F$-vector spaces with the same dimension.
Since $\mathrm{dim}_F\big(\pi(\hat{U}_{n})\big)=\lfloor n/2 \rfloor$ and $\mathrm{dim}_F\big(\pi(\hat{U}_{m})\big)=\lfloor m/2 \rfloor$, we get that $n=m\pm 1$ or $n=m$.
But $n=m\pm 1$ is impossible since the action of $\hat{\Gamma}$ on $\X_{\infty}$ preserves the vertex type.
Thus $n=m$ as desired.
\end{proof}

It follows from \cite[Ch. I, \S 4, Th. 10]{SerreTrees} that $\hat{\Gamma}$ is isomorphic to the sum of $\hat{G}_{-1}$ with the union $\bigcup_{n=0}^{\infty} \hat{G}_{n}$ amalgamated along their common intersection.
Moreover, the union of the group $\hat{G}_n$, for $n\geq 0$, equals:
$$\hat{B}:=\sbmattrix{F^{*}}{S}{J^{-1}}{0}{F^*}{S}{0}{0}{F^{*}}\cap \tilde{\Gamma}=\big\lbrace \U(x,y) \tilde{a}(\lambda) \vert \, (x,y) \in H(L,K)_{S \times J^{-1}}, \,\, \lambda\in F^* \big\rbrace.$$
Therefore, since the intersection of $\hat{B}$ with $\hat{G}_{-1}$ is isomorphic the subgroup $\mathrm{B}(F)$ of upper triangular matrices in $\mathrm{SL}_2(F)$,
next result follows:

\begin{Theorem}\label{teo dec of H1 as an amm product}
The group $\hat{\Gamma}$ is isomorphic to the free product of $\hat{G}_{-1} \cong \mathrm{SL}_2(F)$ with $\hat{B}$, amalgamated by $ \hat{G}_{-1} \cap \hat{B} \cong \mathrm{B}(F)$. \qed
\end{Theorem}

Now we turn to the proof of Th. \ref{main teo hom h1}.
In particular, we now assume $\mathrm{char}(F)=0$.

\begin{proof}[Proof of Th. \ref{main teo hom h1}]
The amalgamated product described in Th. \ref{teo dec of H1 as an amm product} yields a Mayer-Vietoris sequence of the form:
\begin{equation}\label{eq exact sequence for hat Gamma}
\cdots \to H_k\big(\mathrm{B}(F), \mathbb{Z}\big) \to H_k\big(\mathrm{SL}_2(F), \mathbb{Z}\big)\oplus H_k\big(\hat{B},\mathbb{Z}\big) \to H_k\big(\hat{\Gamma}, \mathbb{Z}\big) \to \cdots
\end{equation}
It directly follows from \cite[Prop. 3.2]{Knudson1} that $H^*\big(\mathrm{B}(F), \mathbb{Z}\big) \cong H^*\big(F^*, \mathbb{Z}\big)$.
Moreover, it follows from Proposition \ref{prop inv hom for Bst} applied to $S=F[\sqrt{t}]$ and $T=(1/\sqrt{t})F[\sqrt{t}]$, that $H_*\big(\hat{B}, \mathbb{Z}\big) \cong H_*\big( F^{*}, \mathbb{Z}\big)$.
Hence the exact sequence in \eqref{eq exact sequence for hat Gamma} equals:
\begin{equation}\label{eq exact sequence for hat Gamma 2}
\cdots \to H_k\big(F^*, \mathbb{Z}\big) \to H_k\big(\mathrm{SL}_2(F), \mathbb{Z}\big)\oplus H_k\big(F^*,\mathbb{Z}\big) \to H_k\big(\hat{\Gamma}, \mathbb{Z}\big) \to \cdots
\end{equation}
Therefore, Th. \ref{main teo hom h1} follows as the long exact sequence above \eqref{eq exact sequence for hat Gamma 2} breaks up into short exact sequences of the form $0\to H_k\big(F^*, \mathbb{Z}\big) \to H_k\big(\mathrm{SL}_2(F), \mathbb{Z}\big)\oplus H_k\big(F^*,\mathbb{Z}\big) \to H_k\big(\hat{\Gamma}, \mathbb{Z}\big) \to 0$.
\end{proof}

\section{The homology of the Hecke congruence subgroup \texorpdfstring{$\Gamma_0$}{G0}}\label{section homology of Hecke}

This section is devoted to describing the Hecke congruence subgroup $\Gamma_0=\Gamma \cap \hat{\Gamma}$ as an amalgamated product of simpler subgroups, as well as to understand its homology groups with integer coefficients.
To do so, in \S \ref{subsection fundamental domains} we study a fundamental domain $D$ for the action of $\Gamma_0$  on the Bruhat-Tits tree $\X_{\infty}$.
The method presented here focuses on the description of a certain covering $D'$ of $D$ defined by the action of a principal congruence subgroup on the same tree $\X_{\infty}$.
This method holds even when $\mathrm{char}(F)\neq 0$.
So, in this section we do not assume that $\mathrm{char}(F)=0$, unless we clearly indicate it.

\subsection{Fundamental domain for some congruence subgroups of \texorpdfstring{$\Gamma$}{G}}\label{subsection fundamental domains}

Let $\mathrm{ev}_0: \mathrm{SL}_3(F[\sqrt{t}]) \to \mathrm{SL}_3(F)$ be the group homomorphism induced by the evaluation of $t$ at $0$.
We denote by $\Gamma(t)$ the principal congruence subgroup of $\Gamma=\mathrm{SU}_3(F[t])$ defined as $\Gamma(t):= \ker(\mathrm{ev}_0) \cap \Gamma$.

\begin{Lemma}\label{lemma H0/K0}
One has $\Gamma/\Gamma(t) \cong \mathrm{SO}(q)(F)$.
\end{Lemma}

\begin{proof}
Let $g \in \mathrm{SL}_3(F[\sqrt{t}])$.
Note that $g \in \Gamma$ exactly when $g \Phi g^{*}=\Phi$, with $\Phi$ as in Eq. \eqref{eq of psi} and $g^{*}=\overline{g}^t$ the conjugate transpose of $g$.
In particular $\mathrm{ev}_0(g) \Phi \mathrm{ev}_0(g)^{*}=\Phi$, where $\mathrm{ev}_0(g)^{*}=\mathrm{ev}_0(g)^{T}$.
Then, the image of $\Gamma \to \mathrm{SL}_3(F)$ is the set of matrices in $\mathrm{SL}_3(F)$ preserving $q$.
Thus $\mathrm{Im}\big( \Gamma \to \mathrm{SL}_3(F) \big) \cong \mathrm{SO}(q)(F)$, whence the result follows.
\end{proof}

Note that the quotient $\Gamma_0/\Gamma(t)$ is isomorphic to the image of $\Gamma_0=\Gamma_0(t)$ in $\Gamma/\Gamma(t)$.
This group is $B_0$ as defined in Eq. \eqref{eq B0}.

\begin{Lemma}\label{lemma H01/K0}
One has $\Gamma_0/\Gamma(t) \cong B_0=\left\lbrace \U(x,-x^2/2)\tilde{a}(\lambda) \vert \,  x \in F, \,\, \lambda\in F^{*} \right\rbrace$. \qed
\end{Lemma}

Next result is a technical tool in order to compute the desired fundamental domains.

\begin{Lemma}\label{lemma fundamental region for normal sub}
Let $H$ be a normal subgroup of $G$ such that $1 \to H \to G \xrightarrow{ \mathrm{\pi}} G/H \to 1$ splits.
Let $G'$ be a subgroup of $G$ which is isomorphically mapped to $G/H$ via $\pi$.
Assume that $G$ acts on a tree $X$ via simplicial maps, and let $\mathscr{Y}$ be a fundamental domain for this action. 
If $\mathscr{Z}:=\bigcup_{s \in G'} s \cdot \mathscr{Y}$ is connected, then it is a fundamental domain for the action of $H$ on $X$.
\end{Lemma}

\begin{proof}
Let $\sigma$ be a simplex in $X$, i.e., a vertex or an edge of $X$.
Since $\mathscr{Y}$ is a fundamental region for the action of $G$ on $X$, there exists $g \in G$ and a unique $\sigma_0 \subset \mathscr{Y}$ such that $\sigma = g \cdot \sigma_0$.
By definition of $G' \subset G$, we can decompose $g$ as $g= \gamma s$, where $\gamma \in H$ and $s \in G'$.
Then, we get 
$$\sigma= (\gamma  s) \cdot \sigma_0= \gamma \cdot ( s \cdot \sigma_0),$$
where $s \cdot \sigma_0 \in \mathscr{Z}$.
Since $\gamma\in H$, we conclude that $\mathscr{Z}$ contains a fundamental domain for the action of $H$ on $X$.

Since $\mathscr{Z}$ is connected, it just remains to prove that any two simplices in $\mathscr{Y}$ do not belong to the same $H$-orbit.
Indeed, assume that there exist $s_1,s_2 \in G'$ and $\gamma \in H$ such that $\gamma \cdot (s_1 \cdot \sigma_{0,1})=s_2 \cdot \sigma_{0,2}$, where $\sigma_{0,1}$ and $\sigma_{0,2}$ are two faces in $\mathscr{Y}$. 
Then, the element $g:=s_2^{-1} \gamma s_1\in G$ satisfies $g \cdot \sigma_{0,1}=\sigma_{0,2}$.
Since $\mathscr{Y}$ is a fundamental domain for the action of $G$, we have $\sigma_{0,1}=\sigma_{0,2}$. 
In the sequel, we write $\sigma_0:=\sigma_{0,1}=\sigma_{0,2}$.
Note that $s_1 s_2^{-1} \gamma$ belongs to
$ \mathrm{Stab}_{G}(s_1 \cdot \sigma_0 ) = s_1 \mathrm{Stab}_{G}(\sigma_0) s_1^{-1}.$
Hence, there exists $\kappa \in \mathrm{Stab}_{G}(\sigma_0)$ such that $s_1 s_2^{-1} \gamma = s_1 \kappa s_1^{-1}$.
In particular, we get $\gamma=s_2 \kappa s_1^{-1}$. 

Let $S_{\sigma_0} \subseteq G/H$ be the image of $\mathrm{Stab}_{G}(\sigma_0)$ by the map $\pi: G \to G/H$. 
Note that $\mathrm{Stab}_{G'}(\sigma_0)=G' \cap \mathrm{Stab}_{G}(\sigma_0)$ is isomorphic to $S_{\sigma_0}$ via $\pi$. Then, the exact sequence
$$ 1 \to H \cap \mathrm{Stab}_{G}(\sigma_0) \to \mathrm{Stab}_{G}(\sigma_0) \to S_{\sigma_0} \to 1,$$
is split.
In particular, we can write $\kappa=\kappa_0 p$, where $\kappa_0 \in H$ and $p \in \mathrm{Stab}_{G'}(\sigma_0)$.
Then 
$$\gamma= s_2 \kappa s_1^{-1}=s_2 (\kappa_0 p) s_1^{-1}= s_2 (p s_1^{-1}) (s_1 p^{-1}) \kappa_0 (ps_1^{-1}).$$
Moreover, since $H$ is normal in $G$, we have that $\gamma':=(s_1 p^{-1}) \kappa_0 (ps_1^{-1})$ belongs to $H$.
Therefore, we get $\gamma=s_2 (ps_1^{-1})\gamma'$, or equivalently $\gamma \gamma'^{-1}= s_2 p s_1^{-1}$.
Note that, $\gamma \gamma'^{-1}$ belongs to $H$ while $s_2 p s_1^{-1}$ belongs to $S$.
Since, by definition, we have $H \cap G'= \lbrace \mathrm{id} \rbrace$, we deduce that $\gamma=\gamma'$ and $s_2=s_1 p^{-1}$.
Thus, since $p$ stabilizes $\sigma_0$, we conclude that $s_2 \cdot \sigma_0 = (s_1 p^{-1}) \cdot \sigma_0 = s_1 \cdot \sigma_0$, which concludes the proof.
\end{proof}

\begin{Corollary}\label{coro fund reg for K0}
For each $x\in F$, we write $\mathscr{R}_{x}:=(\U(x,-x^2/2) \mathrm{s}) \cdot \mathscr{R}_{\infty}$.
Then:
\begin{itemize}
\item[(1)] $\mathscr{R}_x \cap \mathscr{R}_y=\lbrace v_0 \rbrace$, for any $(x,y)\in \mathbb{P}^1(F) \times \mathbb{P}^1(F)$ with $x\neq y$, and
\item[(2)] the tree $\mathscr{T}_{\infty}:= \bigcup_{x \in \mathbb{P}^1(F)} \mathscr{R}_{x}$ is a fundamental domain for the action of $\Gamma(t)$ on $\X_{\infty}$.
\end{itemize}
\end{Corollary}

\begin{proof}
Let $(x,y)$ be a pair of different points in $\mathbb{P}^1(F) \times \mathbb{P}^1(F)$ and let $v \in \mathrm{V}(\mathscr{R}_x \cap \mathscr{R}_y)$.
When $x,y\in F$, we write $v=(\U(x,-x^2/2) \mathrm{s}) \cdot v_n$ and $v=(\U(y,-y^2/2) \mathrm{s}) \cdot v_m$, while when $y=\infty$, we just write $v=v_m$.
Since $\U(x,-x^2/2), \U(y,-y^2/2)$ and $\mathrm{s}$ belong to $ \Gamma$, the vertices $v_n$ and $v_m$ belong to the same $\Gamma$-orbit.
Since $\mathscr{R}_{\infty}$ is a fundamental domain for the action of $\Gamma$, we obtain that $v_n=v_m$.
Thus, one of the following conditions holds:
\begin{equation*}
 \mathrm{u}_{-\texttt{a}}(x-y,-(x^2+y^2)/2)= \mathrm{s} \U(x-y,-(x^2+y^2)/2) \mathrm{s}  \in \mathrm{Stab}_{\Gamma}(v_n),
\end{equation*}
when $x,y\in F$, while $\U(y,-y^2/2) \mathrm{s} \in \mathrm{Stab}_{\Gamma}(v_n)$ when $y=\infty$.
Thus, it follows from Lemma \ref{lemma ABLL action} that $n=0$ in either case.
Conversely, since $\U(x,-x^2/2), \U(y,-y^2/2)$ and $\mathrm{s}$ belong to $\mathrm{SO}(q)(F)$, we have that $v_0 \in \mathrm{V}(\mathscr{R}_x \cap \mathscr{R}_y)$.
We conclude that $\mathrm{V}(\mathscr{R}_x \cap \mathscr{R}_y)=\lbrace v_0 \rbrace$.
Since $\mathscr{R}_x \cap \mathscr{R}_y$ is a (full) subgraph of $\X_{\infty}$, we get $\mathscr{R}_x \cap \mathscr{R}_y=\lbrace v_0 \rbrace$.
In particular $\mathscr{T}_{\infty}$ is connected.

In the notation of Lemma \ref{lemma fundamental region for normal sub}, set $G'=\mathrm{SO}(q)(F)$. 
This is a subgroup of $\Gamma$, which is isomorphic to $\Gamma/\Gamma(t)$ according to Lemma \ref{lemma H0/K0}.
Since $\mathscr{R}_{\infty}$ is a fundamental domain for the action of $\Gamma$ on $\X_{\infty}$, it follows from Lemma \ref{lemma fundamental region for normal sub} that $\bigcup_{s \in G'} s \cdot \mathscr{R}_{\infty}$ is a fundamental domain for the action of $\Gamma(t)$ on $\X_{\infty}$.
We obtain from the Bruhat decomposition on $\mathrm{SO}(q)(F)\cong \mathrm{PGL}_2$ that $G'= B_0 \cup U_0 \mathrm{s} B_0$, where $U_0=\lbrace \U(x,-x^2/2) \vert \, x\in F \rbrace$ is the unipotent radical of $B_0$.
Since $B_0$ fixes $\mathscr{R}_{\infty}$, we conclude that $\bigcup_{s \in G'} s \cdot \mathscr{R}_{\infty}=\mathscr{T}_{\infty}$.
\end{proof}

\begin{Corollary}\label{coro fund reg for H01}
The apartment $\mathscr{A}_{\infty}=\mathscr{R}_{\infty} \cup \mathrm{s} \cdot \mathscr{R}_{\infty}$ is a fundamental domain for the action of $\Gamma_0$ on $\X_{\infty}$.
\end{Corollary}

\begin{proof}
Since $\Gamma(t)$ is a normal subgroup of $\Gamma$, the group $\Gamma_0/\Gamma(t) \cong B_0$ acts on $\Gamma(t) \backslash \X_{\infty} \cong \mathscr{T}_{\infty}$ and $\Gamma_0 \backslash \X_{\infty} \cong B_0 \backslash \mathscr{T}_{\infty}$.
Thus, it follows from Corollary \ref{coro fund reg for K0} that $\Gamma_0 \backslash \X_{\infty}\cong \mathscr{R}_{\infty} \cup \mathrm{s} \cdot \mathscr{R}_{\infty}=\mathscr{A}_{\infty}$.
Then, the result follows by lifting the tree $\Gamma_0 \backslash \X_{\infty}$ to the subtree $\mathscr{A}_{\infty}$ of $\X_{\infty}$.
\end{proof}

As in \S \ref{subsection action of G'}, we write $J=\sqrt{t} F[\sqrt{t}]$, and we set:
$$
U_J =  \left \lbrace \U(x,y) \vert \, (x,y) \in H(L,K)_J\right \rbrace, 
$$

\begin{Corollary}\label{coro amal K0}
The group $\Gamma(t)$ is isomorphic to the free product $\Conv_{x \in \mathbb{P}^1(F)} U_J$.
\end{Corollary}

\begin{proof}
Applying Bass-Serre theory \cite[Ch. I, \S 5, Theo. 13]{SerreTrees} on Corollary \ref{coro fund reg for K0}, we obtain that $\Gamma(t)$ is isomorphic to the amalgamated product $\Conv_{\Gamma(t)_{0,0}} \Gamma(t)_x$, where $\Gamma(t)_{0,0}=\mathrm{Stab}_{\Gamma(t)}(v_0)$ and $\Gamma(t)_x$ is the direct limit of the $\Gamma(t)$-stabilizers of vertices in $\mathscr{R}_x$.
For each vertex $v\in \X_{\infty}$ we have $\mathrm{Stab}_{\Gamma(t)}(v)=\Gamma(t) \cap \mathrm{Stab}_{\Gamma}(v)$.
In particular, the stabilizer $\mathrm{Stab}_{\Gamma(t)}(v_0)= \lbrace \mathrm{id} \rbrace$, while for $n>0$ we have:
\begin{equation}
\mathrm{Stab}_{\Gamma(t)}(v_n)=\left \lbrace \U(x,y) \vert \, (x,y) \in H(L,K)_J, \, \, \nu(y)\geq -n/2 \right \rbrace.
\end{equation}
Then $\Gamma(t)_{0,0}=\lbrace \mathrm{id} \rbrace$ and $\Gamma(t)_{\infty}=U_J$.
Moreover, since $\mathscr{R}_{x}=g_x \cdot \mathscr{R}_{\infty}$ with $g_x:=\U(x,-x^2/2) \mathrm{s} \in \Gamma$, we also have $\Gamma(t)_{x}= g_x U_J g_x^{-1} \cong U_J$.
\end{proof}

\begin{Theorem}\label{teo amal dec for hecke}
The group $\Gamma_0$ is isomorphic to the amalgamated product $B \ast_{F^{*}} B$ defined from the injection $F^{*} \hookrightarrow B$ given by $\lambda \mapsto \mathrm{diag}(\lambda,1,\lambda^{-1})$.
\end{Theorem}

\begin{proof}
As in the proof of Prop. \ref{prop fund dom for H1}, let $\mathscr{R}_{\infty,-1}^{\circ}$ be the ray in $\Aa_{\infty}$ whose vertex set is exactly $\lbrace v_n \rbrace_{n=-\infty}^{-1}$.
Let $e$ be the edge of $\X_{\infty}$ connecting $v_{0}$ with $v_{-1}$.
It follows from \cite[Ch. I, \S 5, Theo. 13]{SerreTrees} that $\Gamma_0$ is isomorphic to the amalgamated product $\pi_1 \ast_{\pi_{1,2}} \pi_2$, where $\pi_1$ (resp. $\pi_2$) is the direct limit of the $\Gamma_0$-stabilizers of vertices in $\mathscr{R}_{\infty}$ (resp. $\mathscr{R}_{\infty,-1}^{\circ}$) and $\pi_{1,2}=\mathrm{Stab}_{\Gamma_0}(e)$.
In the notations of \S \ref{subsection quot by arith sub}, it follows from Lemma \ref{lemma ABLL action} that $\mathrm{Stab}_{\Gamma_0}(v_{0})=B_0$.
Moreover, for $n>0$, $\mathrm{Stab}_{\Gamma_0}(v_{n})=G_n$.
Then, we get $\pi_1=B$.
On the other hand, it follows from Lemma \ref{lemma ABLL action} that $\mathrm{Stab}_{\Gamma}(v_{-n})$ equals
$$
\left \lbrace \mathrm{u}_{-\texttt{a}}(x,y) \Tilde{a}(\lambda) \vert \, (x,y) \in H(L,K)_{S}, \,\, \nu(y)\geq -n/2, \,\, \lambda\in F^{*} \right \rbrace.
$$
Then, the group $\mathrm{Stab}_{\Gamma_0}(v_{-n})=\Gamma_0 \cap \mathrm{Stab}_{\Gamma}(v_{-n})$ equals:
$$\left \lbrace \mathrm{u}_{-\texttt{a}}(x,y) \Tilde{a}(\lambda) \vert \, (x,y) \in H(L,K)_J, \,\, \nu(y)\geq -n/2, \,\, \lambda\in F^{*} \right \rbrace.$$
In particular, since $\mathrm{Stab}_{\Gamma_0}(e)=\mathrm{Stab}_{\Gamma_0}(v_0) \cap \mathrm{Stab}_{\Gamma_0}(v_{-1})$, we obtain that $\pi_{1,2}$ equals $\left\lbrace \tilde{a}(\lambda) : \lambda \in F^{*} \right\rbrace \cong F^{*}$.
Now, since $\pi_2$ is the direct limit of $\mathrm{Stab}_{\Gamma_0}(v_{-n})$, $n>0$, we have 
$$\pi_2 = \left \lbrace \mathrm{u}_{-\texttt{a}}(x,y) \Tilde{a}(\lambda) \vert \, (x,y) \in H(L,K)_J, \,\, \lambda\in F^{*} \right \rbrace.$$
Moreover, since the map $\phi: \pi_1 \to \pi_2$, $\U(x,y) \Tilde{a}(\lambda) \mapsto \U(\sqrt{t} x, -t y)\Tilde{a}(\lambda)$ is an isomorphism, we conclude that $\pi_2\cong \pi_1=B$, whence the result follows.
\end{proof}

\subsection{On the homology}\label{subsection hom of hecke}

Next result follows from the Mayer-Vietoris exact sequence defined by the decomposition of $\Gamma(t)$ as a free product given in Corollary \ref{coro amal K0}.

\begin{Proposition}
For each $\Gamma(t)$-module $M$, we have $H_*(\Gamma(t),M) \cong \bigoplus_{x \in \mathbb{P}^1(F)} H_*(U_J,M)$.
\end{Proposition}

\begin{Corollary}
The abelianization $\Gamma(t)^{\mathrm{ab}}$ of $\Gamma(t)$ is isomorphic to $\bigoplus_{x \in \mathbb{P}^1(F)} J$.
\end{Corollary}

\begin{proof}
Since $G^{\mathrm{ab}} =H^1(G,\mathbb{Z})$, we just need to prove that $U_J^{\mathrm{ab}}\cong J$.
Indeed, it is not hard to see that $[\U(u,v), \U(x,y)]=\V(0,u\bar{x}-\bar{u}x)$.
Thus, the commutator $[U_J, U_J]$ is contained in $U_J^{0}:=\left\lbrace \V(x) \vert \, x\in H(L,K)^{0}_J \right\rbrace$.
On the other hand, let $\V(z)\in U_J^{0}$.
By definition of $ U_J^{0}$, we can write $z=\sqrt{t}p(t)$, where $p(t) \in F[t]$.
Therefore $\V(z)=[\U\big(p(t)/2,-\mathrm{N}(p(t))/2\big), \U(-\sqrt{t},-t/2)]$, which implies that $U_J^{0} \subseteq [U_J, U_J]$.
Now, let $f: U_J \to J$ be the map defined by $f\big(\U(x,y)\big)=x$. 
Since $\U(x,y) \U(u,v)=\U(x+v, y+v-\bar{x}u)$, we have that $f$ is a group homomorphism.
Moreover, since $\U(x,-N(x)/2)=x$, the map $f$ is surjective.
Since $\mathrm{ker}(f)=U_J^{0}$, we conclude that $U_{J}^{\mathrm{ab}}=U_J/U_J^{0} \cong J$.
\end{proof}

Analogous results for a more general family of principal congruence subgroups are described in \cite[Cor. 7.8 \& 7.9]{RHBravo}.
Next result, which describes the homology of the Hecke congruence subgroup $\Gamma_0$, follows from the Mayer-Vietoris sequence associated to the amalgamated product in Th. \ref{teo amal dec for hecke}.

\begin{Proposition}\label{prop homology of H01}
For each $\Gamma_0$-module $M$, we have the following exact sequence:
\begin{equation}\label{eq homology of hecke}
\cdots \to H_k(F^{*},M) \to H_k(B,M) \oplus H_k(B,M) \to H_k(\Gamma_0,M) \to \cdots
\end{equation}
\end{Proposition}

In \cite[\S 9.1]{BravoQH}, we prove some results that are analogous to Cor. \ref{coro fund reg for H01} and Prop. \ref{prop homology of H01} in the context of $\mathrm{GL}_2$ and $\mathrm{SL}_2$.

Now, we turn to the proof of Theorem Th. \ref{main teo hom of h01} and Th. \ref{main teo 4}.
In particular, we now assume that $\mathrm{char}(F)=0$.

\begin{proof}[Proof of Th. \ref{main teo hom of h01}.]
It follows from Prop. \ref{prop inv hom for Bst}, applied to $S=T=F[\sqrt{t}]$, that $H_*(B,\mathbb{Z}) \cong H_*(F^{*}, \mathbb{Z})$ when $\mathrm{char}(F)=0$.
Then, the result follows from the exact sequence described in Eq.~\eqref{eq homology of hecke}.
\end{proof}

\begin{proof}[Proof of Th. \ref{main teo 4}.]
Recall that $\tilde{\Gamma}$ is isomorphic to the amalgamated product $\Gamma \ast_{\Gamma_0} \hat{\Gamma}$, according to Lemma \ref{lemma decomposition of H}.
This yields a Mayer-Vietoris sequence with the following form:
\begin{equation}\label{mayer viettoris sequence for H}
\cdots \to H_n\big(\Gamma_0, M\big) \to H_n\big(\Gamma, M\big) \oplus H_n\big(\hat{\Gamma},M\big) \to H_n\big(\tilde{\Gamma}, M\big) \to H_{n-1}\big(\Gamma_0,M\big) \to \cdots
\end{equation}
The next $3$ equations follow respectively from Th. \ref{main teo 1}, Th. \ref{main teo hom h1} and Th. \ref{main teo hom of h01}:
 \begin{align*}
 H_n\big(\Gamma,\mathbb{Z}\big) &\cong H_n\big(\mathrm{PGL}_2(F),\mathbb{Z}\big), \\ 
 H_n\big(\hat{\Gamma}, \mathbb{Z}\big)&\cong H_n\big(\mathrm{SL}_2(F),\mathbb{Z}\big), \\ 
 H_n\big(\Gamma_0,\mathbb{Z}\big) &\cong H_n\big(F^*,\mathbb{Z}\big), \qquad \qquad\forall n \geq 0,
\end{align*}
Thus, the result follows.
\end{proof}

\begin{Corollary}\label{coro ab of tilde gamma}
When $\mathrm{char}(F)=0$, we have that $\mathrm{SU}_3(F[t,t^{-1}])^{\mathrm{ab}}=\lbrace 0 \rbrace$.
\end{Corollary}

\begin{proof}
On one hand, it is well-known that $H_1\big(\mathrm{SL}_2(F), \mathbb{Z}\big)\cong \mathrm{SL}_2(F)^{\mathrm{ab}}=\lbrace 0 \rbrace$.
On the other hand, the group $H_1\big(\mathrm{PGL}_2(F), \mathbb{Z}\big)\cong \mathrm{PGL}_2(F)^{\mathrm{ab}} = \mathrm{PGL}_2(F)/\mathrm{PSL}_2(F) \cong F^*/F^{*2}$.
Then, the Mayer-Vietoris sequence in Eq.~\eqref{mayer viettoris sequence for H}, with values in $M=\mathbb{Z}$, finishes with:
$$ F^* \to F^*/F^{*2} \to \tilde{\Gamma}^{\mathrm{ab}} \to H_0\big(F^*,\mathbb{Z}\big) \to  H_0\big(\mathrm{SL}_2(F), \mathbb{Z}\big)\oplus H_0\big(\mathrm{PSL}_2(F), \mathbb{Z}\big) \to H_0\big(\tilde{\Gamma},\mathbb{Z}\big) \to 0.$$
Since the map $H_0\big(F^{*}, \mathbb{Z}\big) \to H_0\big(\mathrm{SL}_2(F), \mathbb{Z}\big)\oplus H_0\big(\mathrm{PSL}_2(F), \mathbb{Z}\big)$ is injective, we have that $\mathrm{Im}\big( \tilde{\Gamma}^{\mathrm{ab}} \to H_0(F^*,\mathbb{Z} ) \big)=\lbrace 0 \rbrace$. 
Thus, $\mathrm{Ker}\big( \tilde{\Gamma}^{\mathrm{ab}} \to H_0(F^*,\mathbb{Z} ) \big) = \tilde{\Gamma}^{\mathrm{ab}}$.
Hence:
\begin{equation}\label{eq im gamma ab}
\mathrm{Im}\big( F^*/F^{*2} \to \tilde{\Gamma}^{\mathrm{ab}} \big) =\tilde{\Gamma}^{\mathrm{ab}}.
\end{equation}
Note that the group homomorphism $F^* \to F^*/F^{*2}$ is composition $\theta_1 \circ \theta_2 \circ \theta_3$, where $\theta_3$ is the homomorphism $F^* \to \mathrm{PGL}_2(F)$, $\lambda \mapsto \overline{\ssbmatrix{\lambda}{0}{0}{1}}$, the map $\theta_2$ is the projection $\mathrm{PGL}_2(F)\to \mathrm{PGL}_2(F)/\mathrm{PSL}_2(F)$ and $\theta_1$ is the homomorphism $\overline{\det}: \mathrm{PGL}_2(F)/\mathrm{PSL}_2(F)\to F^*/F^{*2}$.
Then, it is straightforward that $F^* \to F^*/F^{*2}$ is surjective.
Thus $\mathrm{ker}\big(  F^*/F^{*2} \to \tilde{\Gamma}^{\mathrm{ab}} \big) = F^*/F^{*2}$, whence the map $ F^*/F^{*2} \to \tilde{\Gamma}^{\mathrm{ab}}$ is zero.
We conclude from Eq. \eqref{eq im gamma ab} that $\tilde{\Gamma}^{\mathrm{ab}}=\lbrace 0 \rbrace$.
\end{proof}

\begin{Lemma}\label{lemma hom PSL2 and SL2}
Let $M$ be a $\mathrm{SL}_2(F)$-module $M$ that is divisible by $2$.
Then, we have $H_*\big( \mathrm{PSL}_2(F), M \big) \cong H_*\big( \mathrm{SL}_2(F), M \big)$.
\end{Lemma}

\begin{proof}
The spectral sequence defined by the exact sequence $1 \to \mu_2(F)=\lbrace \pm 1 \rbrace \to \mathrm{SL}_2(F) \to \mathrm{PSL}_2(F) \to 1$ says that:
$$ E_{p,q}^2=H_p \big( \mathrm{PSL}_2(F), H_q \big( \mu_2(F), M \big) \big) \Rightarrow H_{p+q}\big(\mathrm{SL}_2(F), M\big). $$
Since $ H_q \big( \mu_2(F), M \big) = \lbrace 0 \rbrace$ unless $q=0$, we have $E_{p,q}^2=0$ for all $q>0$.
Thus, $H_*\big( \mathrm{PSL}_2(F), M \big) \cong H_*\big( \mathrm{SL}_2(F), M \big)$ as desired.
\end{proof}

\begin{Corollary}\label{coro alg closed}
Let $F$ be a quadratically closed field with $\mathrm{char}(F)=0$. Then, there is an exact sequence of the form:
$$
\cdots \to H_n\big(F^*, \mathbb{Q} \big) \to H_n\big(\mathrm{SL}_2(F), \mathbb{Q}\big) \oplus H_n\big(\mathrm{SL}_2(F),\mathbb{Q}\big) \to H_n\big(\mathrm{SU}_3(F[t,t^{-1}]), \mathbb{Q}\big) \to \cdots $$
\end{Corollary}

\begin{proof}
Since $F$ is quadratically closed, we have that $\mathrm{PSL}_2(F) \cong \mathrm{PGL}_2(F)$.
Moreover, since $M=\mathbb{Q}$ is divisible, it follows from Lemma \ref{lemma hom PSL2 and SL2} that 
$$H_*\big( \mathrm{PGL}_2(F), \mathbb{Q} \big) \cong H_*\big( \mathrm{PSL}_2(F), \mathbb{Q} \big) \cong H_*\big( \mathrm{SL}_2(F), \mathbb{Q} \big).$$
Since $\mathbb{Q}$ is flat over $\mathbb{Z}$, the Universal coefficient theorem implies that $H_*(G, \mathbb{Q}) \cong H_*(G, \mathbb{Z}) \otimes \mathbb{Q}$, for any group $G$.
Therefore, the identities $H_n\big(\Gamma,\mathbb{Q}\big) \cong H_n\big(\mathrm{PGL}_2(F),\mathbb{Q}\big)$, $
 H_n\big(\hat{\Gamma}, \mathbb{Q}\big) \cong H_n\big(\mathrm{SL}_2(F),\mathbb{Q}\big)$ and $H_n\big(\Gamma_0,\mathbb{Q}\big) \cong H_n\big(F^*,\mathbb{Q}\big)$ follow respectively from Th. \ref{main teo 1}, Th. \ref{main teo hom h1} and Th. \ref{main teo hom of h01}.
 Hence the result follows from Eq. \eqref{mayer viettoris sequence for H}.
\end{proof}

\begin{Example}
Assume that $F=\mathbb{C}$. Since $H_n\big(\mathrm{SL}_2(\mathbb{C}),\mathbb{Q}\big)$ vanishes unless $n\geq 3$, it follows from Cor. \ref{coro alg closed} that the connecting map $\partial: H_n \big(\mathrm{SU}_3 \left(\mathbb{C}\left[t, t^{-1}\right] \right), \mathbb{Q} \big) \to H_{n-1}\big(\mathbb{C}^{*}, \mathbb{Q} \big)$, is injective for $n \geq 4$ and bijective for $n \geq 5$.
\end{Example}

\begin{Lemma}\label{lemma hom of PGL2 and PSL2}
The group $H_n\big(\mathrm{PGL}_2(F), \mathbb{Q})$ is isomorphic to the group $H_n\big(\mathrm{PGL}_2(F), \mathbb{Q})_{F^{*}/F^{*2}}$ of $F^{*}/F^{*2}$-coinvariants elements in $H_n\big(\mathrm{PGL}_2(F), \mathbb{Q})$.
\end{Lemma}

\begin{proof}
Recall that $\mathrm{PGL}_2$ and $\mathrm{PSL}_2$ are related via the exact sequence:
$$ 1 \to \mathrm{PSL}_2(F) \to \mathrm{PGL}_2(F) \to F^*/F^{*2} \to 1, $$
whence we have:
$$ E_{p,q}^2=H_p \big( F^*/F^{*2}, H_q(\mathrm{PSL}_2(F), \mathbb{Q}) \big)\Rightarrow H_{p+q}(\mathrm{PGL}_2(F), \mathbb{Q}).$$
We claim that $E_{p,q}^2= \lbrace 0 \rbrace$ for all $p\neq 0$.
Indeed, since $V:=H_q(\mathrm{PSL}_2(F), \mathbb{Q})$ is a $\mathbb{Q}$-vector space, we have that $H_p(H,V)=0$, for each finite group $H$.
Note that $F^*/F^{*2}$ is an abelian group of exponent $2$ and hence an $\mathbb{F}_2$-vector space.
Then $F^*/F^{*2}$ is the direct limit of finite dimensional $\mathbb{F}_2$-subspaces.
Since homology commutes with direct limits (cf. \cite[\S 5, Ex. 3]{Brown-Cohomology}), we conclude that $E_{p,q}^2=H_p(F^*/F^{*2},V)=\lbrace 0\rbrace$.

Now, it directly follows from the preceding claim that $H_{n}(\mathrm{PGL}_2(F), \mathbb{Q}) \cong E_{0,n}^2$.
In other words, we conclude that $H_{n}(\mathrm{PGL}_2(F), \mathbb{Q}) \cong H_0 \big( F^*/F^{*2}, H_n(\mathrm{PSL}_2(F), \mathbb{Q}) \big) \cong H_n(\mathrm{PSL}_2(F), \mathbb{Q})_{F^*/F^{*2}}$. 
\end{proof}

\begin{Corollary}\label{coro hom SU3(F[t,t-1])}
Let $F$ be a number field.
Let us denote by $r$ (resp. $s$) the number of real (resp. conjugate pairs of complex) embeddings of $F$.
Then, the connecting map:
$$ \partial: H_n\big(\tilde{\Gamma}, \mathbb{Q}\big)= H_n \Big(\mathrm{SU}_3 \left(F\left[t, t^{-1}\right] \right), \mathbb{Q} \Big) \to H_{n-1}\big(F^{*}, \mathbb{Q} \big), $$
is injective for $n \geq 2r+3s+1$ and bijective for $n \geq 2r+3s+2$.
\end{Corollary}

\begin{proof}
It follows from \cite[Proof of Th. 5.1]{Knudson1} that $H_n\big(\mathrm{SL}_2(F),\mathbb{Q}\big)=\lbrace 0 \rbrace$, for all $n \geq 2r+3s+1$.
Then, Lemma \ref{lemma hom PSL2 and SL2} together with Lemma \ref{lemma hom of PGL2 and PSL2} implies that $H_n\big(\mathrm{PGL}_2(F),\mathbb{Q}\big)=\lbrace 0 \rbrace$, for all $n \geq 2r+3s+1$.
Thus, the result follows from Th. \ref{main teo 4}.
\end{proof}



\bibliographystyle{plain}
\bibliography{refs.bib}

\end{document}